\documentclass[a4paper, 12pt]{scrartcl}
\usepackage[T1]{fontenc}
\usepackage[utf8]{inputenc}
\usepackage{paralist}
\usepackage{amsmath}
\usepackage{amsfonts}
\usepackage{amssymb}
\usepackage{amsthm}
\usepackage{amscd}	
\usepackage{tikz-cd}
\usepackage{tikz}
\usetikzlibrary{matrix}
\usepackage{float}
\usepackage{authblk}

\makeatletter
\def\blfootnote{\gdef\@thefnmark{}\@footnotetext}
\makeatother

\vfuzz \hfuzz
\newcommand{\N}{\mathbb{N}}

\newcommand{\Z}{\mathbb{Z}}
\newcommand{\m}{\mathrm{mod} \thinspace}

\newcommand{\homgr}{\mathrm{Hom}_{\mathrm{gr}\Lambda}}

\newcommand{\modgr}{\mathrm{mod} \thinspace \mathrm{gr} \thinspace \Lambda}

\newcommand{\modjgr}{\mathrm{mod}_{J} \thinspace \mathrm{gr} \thinspace \Lambda}

\newcommand{\F}{\mathfrak{F}_{M_{r}D}}
\newcommand{\FX}{\mathfrak{F}_{\overline{X}}}

\newcommand{\FG}{\mathfrak{F}_{\overline{G_r T}}}

\newcommand{\Rad}{\mathrm{Rad}}

\newcommand{\T}{\mathcal{T}}
\newcommand{\G}{\mathcal{G}_{M_r D}}
\newcommand{\GX}{\mathcal{G}_{\overline{X}}}

\newcommand{\GG}{\mathcal{G}_{\overline{G_r T}}}

\newcommand{\id}{\mathrm{id}}
\newcommand{\supp}{\mathrm{supp}}
\newcommand{\soc}{\mathrm{Soc}}
\newcommand{\Top}{\mathrm{Top}}
\newcommand{\Ext}{\mathrm{Ext}}
\newcommand{\Extgr}{\mathrm{Ext}_{\modgr}^{1}}

\newcommand{\GL}{\mathrm{GL}}
\newcommand{\SL}{\mathrm{SL}}
\newcommand{\ZAinf}{\Z [A_{\infty}]}
\newcommand{\ZAdinf}{\Z [A_{\infty}^{\infty}]}
\newcommand{\Sl}{\mathfrak{s}\mathfrak{l}}

\newcommand{\Mat}{\mathrm{Mat}}

\newtheorem{theorem}{Theorem}[section]
\newtheorem{corollary}[theorem]{Corollary}
\newtheorem{lemma}[theorem]{Lemma}

\newtheorem{proposition}[theorem]{Proposition}

\newenvironment{remark}{\textbf{Remark.}}{}
\numberwithin{equation}{section}

\begin{document}
\title{Almost split sequences for polynomial $G_r T$-modules and polynomial parts of Auslander-Reiten components}
\author{Christian Drenkhahn}
\maketitle
\blfootnote{\textup{2010} \textit{Mathematics Subject Classification}:
16G70, 20M32, 14L15}
\begin{abstract}
\textbf{Abstract.} In \cite{DNP1}, Doty, Nakano and Peters defined infinitesimal Schur algebras, combining the approach via polynomial representations with the approach via $G_r T$-modules to representations of the algebraic group $G = \GL_n$. We study analogues of these algebras and their Auslander-Reiten theory for reductive algebraic groups $G$ and Borel subgroups $B$ by considering the categories of polynomial representations of $G_r T$ and $B_r T$ as full subcategories of $\m G_r T$ and $\m B_r T$, respectively. We show that every component $\Theta$ of the stable Auslander-Reiten quiver $\Gamma_s(G_r T)$ of $\m G_r T$ whose constituents have complexity 1 contains only finitely many polynomial modules. For $G = \GL_2, r = 1$ and $T \subseteq G$ the torus of diagonal matrices, we identify the polynomial part of the stable Auslander-Reiten quiver of $G_r T$ and use this to determine the Auslander-Reiten quiver of the infinitesimal Schur algebras in this situation. For the Borel subgroup $B$ of lower triangular matrices of $\GL_2$, the category of $B_r T$-modules is related to representations of elementary abelian groups of rank $r$. In this case, we can extend our results about modules of complexity $1$ to modules of higher Frobenius kernels arising as outer tensor products. 
\end{abstract}

\section*{Introduction}
Let $G = \GL_n$ and $T \subseteq G$ the torus of diagonal matrices over an algebraically closed field of characteristic $p > 0$. In \cite{DNP1}, Doty, Nakano and Peters investigated polynomial representations of the group scheme $G_r T$. These can be considered as modules over certain finite-dimensional algebras, the infinitesimal Schur algebras $S_d (G_r T)$. Using the general framework of \cite{Doty1}, their definition can be extended to other group schemes $G$.  In \cite{Farn2}, the Auslander-Reiten theory of $\m G_r T$ was studied by considering $G_r T$-modules as $X(T)$-graded $G_r$-modules, where $X(T) \cong \Z^n$ is the character group of $T$. We follow this approach and study, for a component $\Theta$ of the stable Auslander-Reiten quiver $\Gamma_s(G_r T)$ of $\m G_r T$, the polynomial part of $\Theta$, that is the intersection of $\Theta$ with the full subcategory of polynomial $G_r T$-modules. Here, $G \subseteq \GL_n$ is either reductive or a Borel subgroup of a reductive group containing the center $Z(\GL_n)$ of $\GL_n$. The number of elements of the polynomial part of $\Theta$ is shown to be finite for components of complexity one. Since $\m G_r T$ is a Frobenius category and the category of polynomial representations is not, the polynomial part of $\Theta$ also contains natural types of $G_r T$-modules whose position in $\Theta$ one can determine, namely modules which are $\Ext$-projective or have finite projective dimension in the full subcategory of polynomial $G_r T$-modules. Furthermore, the polynomial part of a component is part of the Auslander-Reiten quiver of an infinitesimal Schur algebra. We use this approach to fully determine the Auslander-Reiten quiver of the algebras $S_d(G_1 T)$ for $G = \GL_2$ and $T \subseteq G$ the torus of diagonal matrices. It turns out that upon deleting projective-injective modules, the underlying directed graphs of all components in this situation are isomorphic to $\Z[A_{2s+1}]/\langle \tau^{2s+1}\rangle$ for some $s \in \N$. By reducing to the case $r = 1$, we are able to determine the polynomial part of all $G_r T$-components of type $\ZAdinf$ whose restriction to $G_r$ has type $\Z[\tilde{A}_{12}]$ for higher $r$. This paper is organized as follows.\\
In the first section, we collect basic results about $\Z^n$-graded algebras and show how $G_r T$-modules can be considered as $X(T)$-graded $G_r$-modules. \\
In the second section, we consider a subcategory $\T$ of the category $\modgr$ of $\Z^n$-graded modules over a $\Z^n$-graded algebra $\Lambda$, where $\T$ is closed with respect to some natural operations. We prove results about the Auslander-Reiten theory of $\T$ valid in this context. Here, we concentrate on the intersection of $\T$ with components of type $\ZAinf$. By \cite[Proposition 8.2.2, Theorem 8.2.3]{Farn2}, components of type $\ZAinf$ are the preeminent components of the stable Auslander-Reiten quiver of $\m G_r T$ for reductive groups $G$.\\
In Section 3, we define polynomial representations of $G_r T$ and infinitesimal Schur algebras and show how they fit into the general framework of Section 2. We then prove results about the position of special kinds of modules in the Auslander-Reiten quiver and give some criteria on modules in components to ensure that the polynomial part of these components is finite.\\
In Section 4, we turn to modules of complexity $1$. In this context, we can combine results from \cite{Farn2} and the previous sections to show that the polynomial parts of their stable AR-components are finite. For $G = \GL_n$ or $G \subseteq \GL_n$ the Borel subgroup of lower triangular matrices, we show that these components contain a unique polynomial $G_r T$-module of maximal quasi-length, provided they contain a polynomial module whose quasi-length is large enough. \\
In Section 5, we completely determine the polynomial parts of components of $\m G_1 T$ for $G = \GL_2$ and use this to determine the Auslander-Reiten quiver of the infinitesimal Schur algebras $S_d(G_1 T)$. For this, we rely heavily on the classification of indecomposable finite-dimensional $U_0(\Sl_2)$-modules obtained by Premet in \cite{Pre1}. By adapting Morita-equivalences between blocks of $G_{r-s} T$ and $G_r T$ to the setting of polynomial representations of $\GL_n$, we are able to use our results in this case to obtain the aforementioned result for $G_r T$-components of type $\ZAdinf$ whose restriction to $G_r$ has type $\Z[\tilde{A}_{12}]$. \\
In Section 5, we give further results on polynomial representations of $B_r T$, where $B$ is a Borel subgroup of a reductive group and $T \subseteq B$ a maximal torus. We show that in contrast to the case pertaining to $\GL_n$ (cf. \cite[Section 7]{DNP2}), the algebras $S_d(B_r T)$ are directed quasi-hereditary algebras. For $B \subseteq \GL_2$ the Borel subgroup of lower triangular matrices, we can extend some results about modules of complexity 1 to modules which are outer tensor products of two modules, one of which has complexity 1. This case is of interest since representations of $B_r T$ can be viewed as graded modules over the truncated polynomial ring $k[X_1, \ldots, X_r]/(X_1^p, \ldots, X_r^p)$, linking them to representations of elementary abelian $p$-groups of rank $r$. \\
For representations of algebraic groups and $G_r T$-modules, we refer the reader to \cite{Jantz1}. Details about Auslander-Reiten theory and representations of associative algebras can be found in \cite{ASS1}, \cite{ARS1}.

\section{$\Z^n$-graded algebras and $G_r T$-modules}
\label{section: GrT - modules and graded algebras}
In this section, we are going to establish basic notation and results for $\Z^n$-graded algebras and $G_r T$-modules.\\
Let $k$ be a field and $\Lambda$ be a finite-dimensional $k$-algebra.
The algebra $\Lambda$ is called $\Z^n$-graded if there is a decomposition

\begin{equation*}
\Lambda = \bigoplus_{i \in \Z^n} \Lambda_i
\end{equation*}

into $k$-subspaces such that $\Lambda_i \Lambda_j \subseteq \Lambda_{i+j}$ for all $i,j \in \Z^n$.
Letting $\m \Lambda$ be the category of finite-dimensional $\Lambda$-modules, $M \in \m \Lambda$ is called $\Z^n$-graded if there is a decomposition

\begin{equation*}
M = \bigoplus_{i \in \Z^n} M_i
\end{equation*}
such that $\Lambda_i M_j \subseteq M_{i+j}$ for all $i, j \in \Z^n$. An element $m \in M\setminus\{0\}$ is called homogeneous of degree $i \in \Z^n$ if $m \in M_i$. In that case, we write $\deg(m) = i$ for the degree of $m$. A $\Lambda$-submodule $N \subseteq M$ is called homogeneous if 
\begin{equation*}
N = \bigoplus_{i \in \Z^n} N \cap M_i.
\end{equation*}
In that case, $N$ and the factor module $M / N$ have a natural structure as $\Z^n$-graded $\Lambda$-modules.
If $M = \bigoplus_{i \in \Z^n} M_i, N = \bigoplus_{i \in \Z^n} N_i$ are $\Z^n$-graded modules, then a morphism of $\Z^n$-graded modules $M \rightarrow N$ is a $\Lambda$-linear map $f: M \rightarrow N$ such that $f(M_i) \subseteq N_i$ for all $i \in \Z^n$. We denote the space of $\Z^n$-graded morphisms $M \rightarrow N$ by $\homgr(M, N)$ and the category of finite dimensional $\Z^n$-graded $\Lambda$-modules by $\modgr$.

We denote by $F: \modgr \rightarrow \m \Lambda$ the forgetful functor. The modules in $F(\modgr)$ are called gradable. 
For $i \in \Z^n$, there is a functor $[i]: \modgr \rightarrow \modgr$, where $M[i]_j := M_{j-i}$ for all $i \in \Z^n$ and all $M \in \modgr$ and the morphisms are left unchanged. Then $[i]$ is an auto-equivalence of $\modgr$ and we have $F \circ [i] = F$ for all $i \in \Z^n$.\\
The following results about $\Z^n$-graded algebras were first obtained  in \cite{GoGr2}, \cite{GoGr1} for $n = 1$ and it has been known for a long time that they hold for $n > 1$. Proofs for $n > 1$ can be found in \cite{Dre2}.
\begin{proposition}
\begin{enumerate}[(1)]
\item A module $M \in \modgr$ is indecomposable iff $F(M)$ is indecomposable in $\m \Lambda$.
\item If $M, N \in \modgr$ are indecomposable such that $F(M) \cong F(N)$, there is a unique $\lambda \in \Z^n$ such that $M \cong N[\lambda]$.
\item If $(P_n)_{n \in \N}$ is a minimal projective resolution of $M$, then $(F(P_n))_{n \in \N}$ is a minimal projective resolution of $F(M)$.
\end{enumerate}
\label{prop: indecomposables, projectives in modgr}
\end{proposition}

\begin{theorem}
\begin{enumerate}[(1)]
\item The category $\modgr$ has almost split sequences.
\item If $0 \rightarrow M \rightarrow E \rightarrow N \rightarrow 0$ is an almost split sequence in $\modgr$, then $0 \rightarrow F(M) \rightarrow F(E) \rightarrow F(N) \rightarrow 0$ is almost split in $\m \Lambda$.
\end{enumerate}
\label{thm: modgr has almost split sequences}
\end{theorem}

For $M \in \modgr$, we define the support of $M$ as $\supp(M) = \{\lambda \in \Z^n \mid M_{\lambda} \neq 0\}$.
If $J \subseteq \Z^n$, we denote by $\modjgr$ the full subcategory of $\modgr$ whose objects are those $M \in \modgr$ such that $\supp(M) \subseteq J$. The following result is useful for extending results about $\m \Lambda$ to $\modgr$.

\begin{proposition}
Let $J \subseteq \Z^n$ be finite. Then there is a finite-dimensional algebra $A$ and an equivalence of categories $\modjgr \rightarrow \m A$.
\label{prop: modjgr equivalent to mod A}
\end{proposition}

Now let $k$ be an algebraically closed field of characteristic $p > 0$.
We want to apply results about $\Z^n$-graded algebras to the category $\m G_r T$, where $G$ is a smooth connected algebraic group scheme over $k$, $T \subseteq G$  is a maximal torus and $G_r$ the $r$-th Frobenius kernel of $G$. The next proposition collects some results on this situation in a slightly more general context.  
\begin{proposition}
Let $H$ be an affine group scheme, $G, T \subseteq H$ be closed subgroup schemes such that $T$ is a torus and $G$ is infinitesimal and normalized by $T$. Then the following statements hold:
\begin{enumerate}[(1)]
\item The category $\m G \rtimes T$ is equivalent to the category of $X(T)$-graded $G$-modules.
\item The category $\m G T$ is a sum of blocks of $\m G \rtimes T$. If $M \in \m G T, \lambda \in X(T)$, then $M[\lambda] \in \m G T$ iff $\lambda$ is trivial on $G \cap T$.  
\item The forgetful functor induces a morphism of stable translation quivers $\Gamma_s (G T) \rightarrow \Gamma_s (G)$.
\end{enumerate} 
\label{prop: mod HT sum of blocks of mod H rtimes T}
\end{proposition}

\begin{proof}
This can be proved as \cite[2.1]{Farn4}. 
\end{proof}

\section{Almost split sequences in subcategories}
In this section, we want to establish basic results on almost split sequences and Auslander-Reiten theory in subcategories. This allows us to treat some results in the following sections in a uniform way. Let $\Lambda$ be a finite-dimensional $\Z^n$-graded algebra and $\T, \mathcal{F}$ be full subcategories of $\modgr$ closed with respect to finite direct sums. 
Recall that $(\T, \mathcal{F})$ is called a torsion pair if $\T = \{V \in \modgr \mid \homgr(V, N) = 0$ for all $N \in \mathcal{F} \}$ and $\mathcal{F} = \{V \in \modgr \mid \homgr(N, V) = 0$ for all $N \in \T \}$. In that case, $\T$ is called a torsion class and $\mathcal{F}$ is called a torsion-free class in $\modgr$.  By \cite[VI.1.4]{ASS1}, $\T$ is a torsion class if and only if $\T$ is closed under images, finite direct sums and extensions and $\mathcal{F}$ is a torsion-free class if and only if $\mathcal{F}$ is closed under submodules, direct products and extensions. This is also equivalent to the existence of a subfunctor $t$ of the identity functor $\modgr \rightarrow \modgr$ called the torsion radical such that $t \circ t = t$ and $\T = \lbrace M \mid tM = M\rbrace$, $\mathcal{F} = \lbrace M \mid tM = 0\rbrace$. Additionally, we can define a functor $u: \modgr \rightarrow \modgr$ mapping $M \in \modgr$ to $M / tM$, the largest factor module of $M$ belonging to $\mathcal{F}$. \\
Since not all subcategories we consider in later sections are extension-closed, we formulate the results of this section for more general subcategories $\T, \mathcal{F}$ whenever possible.  

We will need the following lemma and a dual version which were proved for $\m \Lambda$ in \cite[Lemma 2, Lemma 3]{Ho1}. They can be translated to our setting using \ref{prop: modjgr equivalent to mod A}.  We say that $M \in \T$ is $\Ext$-projective resp. $\Ext$-injective in $\T$ if $\Extgr(M,-)\vert_{\T} = 0$ resp. $\Extgr(-, M)\vert_{\T}=0$.
\begin{lemma}
Let $(\T, \mathcal{F})$ be a torsion pair in $\modgr$ and $V \in \T$ be indecomposable.
\begin{enumerate}[(1)]
\item The module $V$ is $\Ext$-projective in $\T$ iff $\tau_{\modgr} (V) \in \mathcal{F}$. 
\item Suppose $V$ is not $\Ext$-projective. Then $t(\tau_{\modgr}(V))$ is indecomposable and if $0 \rightarrow \tau_{\modgr}(V) \stackrel{f}{\rightarrow} E \stackrel{g}{\rightarrow} V \rightarrow 0$ is the almost split sequence in $\modgr$ ending in $V$, then the induced sequence $0 \rightarrow t(\tau_{\modgr}(V)) \rightarrow t(E) \rightarrow V \rightarrow 0$ is the almost split sequence in $\T$ ending in $V$.
\end{enumerate}

\label{lem: Ext-projective in T for graded algebras, almost split sequences in T}
\end{lemma}

\begin{proof}
\begin{enumerate}[(1)]
\item  Let $V \in \T$ be $\Ext$-projective and let $J = \supp(V) \cup \supp(\tau_{\modgr}(V))$. Then $J$ is finite and $(\T \cap \modjgr, \mathcal{F} \cap \modjgr)$ is a torsion pair in $\modjgr$. As $V$ is $\Ext$-projective in $\T \cap \modjgr$, a consecutive application of \ref{prop: modjgr equivalent to mod A} and \cite[Lemma 2]{Ho1} implies $\tau_{\modgr}(V) = \tau_{\modjgr}(V) \in \mathcal{F} \cap \modjgr$.\\
For the other direction, let $\tau_{\modgr}(V) \in \mathcal{F}$ and $W \in \T$. Let $J = \supp(V) \cup \supp(\tau_{\modgr}(V)) \cup \supp(W)$. Then all extensions of $V$ by $W$ belong to $\modjgr$ and $\tau_{\modjgr}(V) \in \mathcal{F} \cap \modjgr$. Thus, \cite[Lemma 2]{Ho1} yields $\Extgr(V, W) = \Ext_{\modjgr}^1(V, W) = 0$, so that $V$ is $\Ext$-projective in $\T$. 
\item Let $X \in \T$ and $\alpha: X \rightarrow V$ not a split epimorphism. Set $J = \supp(\tau_{\modgr}(V)) \cup \supp(V) \cup \supp(X)$. Then $0 \rightarrow t(\tau_{\modgr}(V)) \rightarrow t(E) \rightarrow V \rightarrow 0$ is the image of $0 \rightarrow \tau_{\modgr}(V) \rightarrow E \rightarrow V \rightarrow 0$ under the torsion radical of $\T \cap \modjgr$, so that the sequence is almost split in $\T \cap \modjgr$ by \cite[Lemma 2]{Ho1}. As $\alpha$ is not a split epimorphismus in $\T \cap \modjgr$, $\alpha$ factors through $t(g)$, so that $t(g)$ is right almost split in $\T$. Analogously, one shows that $t(f)$ is left almost split in $\T$, so that the sequence is almost split in $\T$.
\end{enumerate}
\end{proof}

If $\T$ is closed with respect to submodules and factor modules, we can define functors $t, u: \modgr \rightarrow \modgr$  by letting $t(V)$ be the largest submodule and $u(V)$ be the largest factor module in $\T$ of $V \in \modgr$. Then standard arguments show that $t$ is right adjoint to the inclusion functor $\T \rightarrow \modgr$, so that $t$ is left exact and maps injectives in $\modgr$ to injective objects in $\T$, while $u$ is  left adjoint to the inclusion functor, so that $u$ is right exact and maps projectives in $\modgr$ to projective objects in $\T$. Thus, $\T$ has enough projectives and enough injectives in this case and the notions of $\Ext$-projective resp. -injective and projective resp. injective object in $\T$ coincide. The notions of top, socle and radical in $\modgr$ and $\T$ also coincide in this case and $t$ resp. $u$ maps the injective envelope resp. projective cover of $V \in \modgr$ to the injective envelope of $t(V)$ resp. projective cover of $u(V)$ in $\T$. We call $V \in \modgr$ $t$-acyclic resp. $u$-acyclic if the higher right derived functors of $t$ resp. left derived functors of $u$ vanish on $V$. If $M \in \modgr$ is $t$-acyclic, we can compute minimal injective resolutions and Heller shifts in $\T$ from those in $\modgr$.  
\begin{proposition}
Suppose that $\T$ is closed with respect to submodules and factor modules. Let $V \in \T$ be $t$-acyclic.
\begin{enumerate}[(1)]
\item If $0 \rightarrow V \rightarrow I_0 \rightarrow I_1 \rightarrow \ldots$ is a minimal injective resolution of $V$ in $\modgr$, then $0 \rightarrow V \rightarrow t(I_0) \rightarrow t(I_1) \rightarrow \ldots$ is a minimal injective resolution of $V$ in $\T$.
\item We have $\Omega_{\T}^{-i}(V) = t(\Omega_{\modgr}^{-i}(V))$ for all $i \in \N_0$.
\end{enumerate}
\label{prop: M t-acyclic, I min res => t(I) min res}
\end{proposition}

\begin{proof}
This can be proved with the same arguments as \cite[Theorem 7.3]{DNP2}, noting that the authors use positive superscripts instead.
\end{proof}
We leave it to the reader to formulate a dual version about projective resolutions and positive Heller shifts for a $u$-acyclic $V \in \T$.
Recall that the injective resp. projective dimension of a module $V$ is the length of a minimal injective resp. projective resolution of $V$, i.e. the smallest natural number $i$ such that $\Omega^{-(i+1)}(V) = 0$ resp. $\Omega^{i+1}(V) = 0$. We write $id_{\T}(V)$ resp. $pd_{\T}(V)$ for the injective resp. projective dimension of $V \in \T$. 

\begin{corollary} 
Suppose $\T$ is closed with respect to submodules and factor modules. Let $M \in \T$ be $t$-acyclic such that $s = id_{\T}(M) < \infty$. Then  $\T \cap \lbrace \Omega_{\modgr}^{-i}(M) \mid i \in \N_0 \rbrace$ is finite and all elements of this set have finite injective dimension in $\T$.
\label{cor: M t-acyclic with finite injective dimension => only finitely many T-modules in negative Omega-orbit}
\end{corollary}

\begin{proof}
By \ref{prop: M t-acyclic, I min res => t(I) min res}.(2), we have $t(\Omega_{\modgr}^{-l}(M)) = 0$ for $l > s$, so that $\Omega^l_{\modgr}(M) \notin \T$ or $\Omega^l_{\modgr}(M) = 0$. Since minimal injective resolutions for $M$ in $\T$ induce minimal injective resolutions for $t(M)$ in $\T$, this also shows that all modules in $\T \cap \lbrace \Omega_{\modgr}^{-i}(M) \mid i \in \N_0 \rbrace$ have finite injective dimension in $\T$.   
\end{proof}

In the remainder of this section, let $\Theta$ be a regular component of type $\ZAinf$ of the Auslander-Reiten quiver $\Gamma(\modgr)$ of $\modgr$. By definition, $\Theta$ does not have any projective vertices and consequently, the arrows of $\Theta$ pointing downwards correspond to irreducible epimorphisms and the arrows pointing upwards correspond to irreducible monomorphisms. The next two results determine the position of $\Ext$-projective modules in $\T$ inside $\Theta$. Recall that $V \in \Theta$ is called quasi-simple if it belongs to the bottom layer of $\Theta$. In our context, this means that the middle term of the almost split sequence starting in $V$ is indecomposable. For each $M \in \Theta$, there is a unique quasi-simple module $N \in \Theta$ such that $N \subseteq M$, the quasi-socle of $M$.  A path inside $\Theta$ is a sectional path if no vertex on the path is a $\tau_{\modgr}$-shift of another vertex on the path. For a regular component of type $\ZAinf$, this is equivalent to all arrows being surjective or all arrows being injective. We denote by $ql(M)$ the quasi-length of $M \in \Theta$, that is the number of vertices on a sectional path from the quasi-socle of $M$ to $M$. Thus, if $M$ is quasi-simple, we have $ql(M) = 1$.  

\begin{proposition}
Suppose $\T$ is a torsion class and closed with respect to submodules. Let $M \in \Theta \cap \T$ such that $M$ is $\Ext$-projective in $\T$ and $N$ be the quasi-socle of $M$. Then all modules on the sectional path from $N$ to $M$ are $\Ext$-projective in $\T$. If $ql(M)>1$, then $N$ is simple.
\label{prop: projective $T$-modules in ZAinfty components of modgr}
\end{proposition}

\begin{proof}
By \ref{lem: Ext-projective in T for graded algebras, almost split sequences in T}, $\tau_{\modgr}(M)$ has no nontrivial $\T$-submodules. If $M'$ is a predecessor of $M$ on the sectional path from $N$ to $M$, then $\tau_{\modgr}(M')$ has no nontrivial $\T$-submodules as it embeds into $\tau_{\modgr}(M)$. By \ref{lem: Ext-projective in T for graded algebras, almost split sequences in T}, $M'$ is $\Ext$-projective in $\T$. Now let $ql(M)>1$. Then there is an almost split sequence in $\T$ starting in $N$ such that the middle term of the sequence is $\Ext$-projective in $\T$. We use arguments dual to those of \cite[V.3.3]{ARS1} to show that $N$ is simple. Let $B$ be the middle term and $C$ the right term of the almost split sequence. Without loss of generality, we may assume that $N$ is a submodule of $B$. Since the sequence is almost split, $N$ is a proper submodule of $B$. As $B$ is projective indecomposable in $\T$, this implies $N \subseteq \Rad(B)$. Thus, the sequence $0 \rightarrow \Top(N) \rightarrow \Top(B) \rightarrow \Top(C) \rightarrow 0$ is not exact. Hence a dual version of \cite[V.3.2]{ARS1} shows that $N$ is simple. 
\end{proof}

\begin{corollary}
Suppose $\T$ is a torsion class and closed with respect to submodules.
Let $M \in \Theta \cap \T$ such that $M$ is $\Ext$-projective in $\T$. Suppose there is a duality $D: \T \rightarrow \T$ such that $D(S)\cong S$ for every simple module $S \in \T$. Then $M$ is quasi-simple.  
\label{cor: projective $T$-modules in ZAinfty components of modgr with duality fixing the simples}
\end{corollary}

\begin{proof}
If $ql(M)>1$, \ref{prop: projective $T$-modules in ZAinfty components of modgr} provides a simple module $S \in \T$ which is $\Ext$-projective in $\T$. Since $D(S)\cong S$, the remarks preceding \ref{prop: M t-acyclic, I min res => t(I) min res} yield that $S$ is also $\Ext$-injective in $\T$, so the almost split sequence starting in $S$ splits, a contradiction.
\end{proof}

\begin{lemma}
Let $\T$ be closed with respect to submodules and factor modules. Suppose the number of quasi-simple modules in $\Theta \cap \T$ is finite. Then $\Theta \cap \T$ is finite.
\label{lem: only finitely many quasi-simple modules in theta cap mod MrD => theta cap mod MrD finite}
\end{lemma}

\begin{proof}
Since $\T$ is closed with respect to taking submodules and factor modules, our assumption implies that the number of modules of any given quasi-length in $\Theta \cap \T$ is finite and that the quasi-length of modules in $\Theta \cap \T$  is bounded. Hence $\Theta \cap \T$ is finite.
\end{proof}

For $M \in \Theta$ of quasi-length $m$, we denote the quasi-socle of $M$ by $M(1)$ and the modules on the sectional path from $M(1)$ to $M$ by $M(1), M(2), \ldots, M(m)= M$. We denote by $\mathcal{W}(M)$ the wing of $M$, namely the mesh-complete full subquiver of $M$ containing the vertices $\tau_{\modgr}^r(M(s))$ with $1 \leq s \leq m$ and $0 \leq r \leq m-s$, see \cite[3.3]{R1} and the following figure.
\begin{figure}[h]
\[ \begin{picture}(400, 160)
\multiput(40, 140)(40, 0)7{$\bullet$}

\multiput(20, 120)(40, 0)7{$\bullet$}

\multiput(40, 100)(40, 0)7{$\bullet$}

\multiput(20, 80)(40, 0)7{$\bullet$}

\multiput(40, 60)(40, 0)7{$\bullet$}

\multiput(20, 40)(40, 0)7{$\bullet$}

\multiput(40, 20)(40, 0)7{$\bullet$}

\multiput(20, 0)(40, 0)7{$\bullet$}


\multiput(45, 140)(40, 0)6 {\vector(1, -1){15}}

\multiput(25, 125)(40, 0)7 {\vector(1, 1){15}}

\multiput(25, 120)(40, 0)7 {\vector(1, -1){15}}

\multiput(45, 105)(40, 0)6 {\vector(1, 1){15}}

\multiput(45, 100)(40, 0)6 {\vector(1, -1){15}}

\multiput(25, 85)(40, 0)7 {\vector(1, 1){15}}

\multiput(25, 80)(40, 0)7 {\vector(1, -1){15}}

\multiput(45, 65)(40, 0)6{\vector(1,1){15}}

\multiput(45, 60)(40, 0)6 {\vector(1, -1){15}}

\multiput(25, 45)(40, 0)7{\vector(1, 1){15}}

\multiput(25, 40)(40, 0)7{\vector(1, -1){15}}

\multiput(45, 25)(40, 0)6{\vector(1,1){15}}

\multiput(45, 20)(40, 0)6{\vector(1,-1){15}}

\multiput(25, 5)(40, 0)7{\vector(1, 1){15}}


\multiput(40, 160)(40, 0)7{\vdots}

\multiput(0, 140)(0, -20)8{\ldots}
\multiput(305, 140)(0, -20)8{\ldots}

{\color{red}
\put(140,80){$\bullet$}
\multiput(120, 60)(40, 0)2{$\bullet$} 
\multiput(100, 40)(40, 0)3{$\bullet$}
\multiput(80, 20)(40, 0)4{$\bullet$}
\multiput(60, 0)(40, 0)5{$\bullet$}

}
\put(140, 95){$M$}

\end{picture}\]
\centering \caption{The red vertices in this $\ZAinf$-component form the wing of the top red vertex $M$.}

\end{figure}
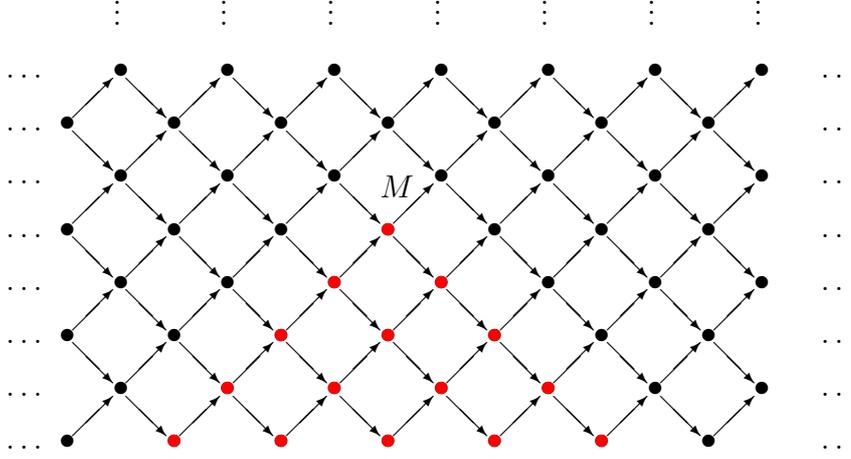
\begin{lemma}
Suppose $\T$ is closed with respect to submodules and factor modules and let $M \in \Theta \cap \T$. Then every module in the wing of $M$ belongs to $\T$.
\label{lem: mountains in ZAinfty} 
\end{lemma}

\begin{proof}
We show this by induction on the quasi-length $ql(M)$ of $M$. If $ql(M) = 1$, i.e. $M$ is quasi-simple, the statement is clear. Now let $m = ql(M) > 1$. Since $\T$ is closed with respect to submodules and factor modules, $M(m-1), \tau_{\modgr}^{-1}(M(m-1)) \in \T$. As the remaining vertices in $\mathcal{W}(M)$ belong to $\mathcal{W}(M(m-1)) \cup \mathcal{W}(\tau^{-1}_{\modgr}(M(m-1)))$, the result now follows by induction.
\end{proof}
 
The assumptions on $M$ in the following lemma mean that two wings in $\T \cap \Theta$ have adjacent quasi-simple modules. We use this to show that the modules between the wings also belong to $\Theta$ and construct a module $N \in \T \cap \Theta$ such that the wing of $N$ contains the wing of $M$ and $\tau_{\modgr}^s(M)$.
   
\begin{lemma}
Suppose $\T$ is a torsion class closed with respect to submodules. Let $s \in \N$ and $M \in \Theta \cap \T$ such that $2ql(M)> s$ and $\tau_{\modgr}^{s}(M)\in \T$. Then $\tau_{\modgr}^{i}(M) \in \T$ for $1 \leq i \leq s$ and there exists a module $N \in \Theta \cap \T$ such that $ql(N) = ql(M) + s$. 
\label{lem: M, tau^i M in mod MrD for ql(M) large enough}  
\end{lemma}

\begin{proof}
As $ql(M) = ql(\tau_{\modgr}^{s}(M)) > \frac{s}{2}$, the wings of $\tau_{\modgr}^{s}(M)$ and $M$ either intersect or the rightmost quasi-simple module $Y_1$ in the wing of $\tau_{\modgr}^{s}(M)$ and the leftmost quasi-simple module $Y_2$ in the wing of $M$ satisfy $\tau_{\modgr}(Y_2) = Y_1$. As $\T$ is extension-closed, this shows that $\tau_{\modgr}^i (M) \in \T$ for $1 \leq i \leq s$. \\
Using again that $\T$ is extension-closed, we get that in the layer above a horizontal line $A, \tau_{\modgr} (A),\ldots , \tau_{\modgr}^{l} (A)$ of $(l + 1)$ modules in $\T$, there is a horizontal line of the same form of $l$ modules in $\T$. As the horizontal line in our situation has at least $(s + 1)$ modules in $\T$, the existence of $N$ now follows inductively.   
\end{proof}

Recall that a homogeneous tube is a component isomorphic to $\ZAinf / \langle \tau \rangle$. If $\Theta$ becomes a homogeneous tube upon forgetting the grading, we can determine when the almost split sequence in $\T$ ending in $N \in \T \cap \Theta$ is also almost split in $\modgr$. 

\begin{proposition}
Let $F: \modgr \rightarrow \m \Lambda$ be the forgetful functor and $\T$ be a torsion class. Suppose that $F(\Theta)$ is a homogeneous tube and $N \in \Theta \cap \T$ is not $\Ext$-projective in $\T$. Let
\begin{equation*}
(0) \rightarrow V \rightarrow E \rightarrow N \rightarrow (0)
\end{equation*}
be the almost split sequence in $\T$ ending in $N$. Then the sequence is almost split in $\modgr$ if and only if $l(V) = l(N)$ if and only if $\dim_k V = \dim_k N$.
\end{proposition}

\begin{proof}
By \ref{lem: Ext-projective in T for graded algebras, almost split sequences in T}, $V$ is a submodule of the left term of the almost split sequence in $\modgr$ ending in $V$. Since $F(\Theta)$ is a homogeneous tube, the left term of that sequence has the same length and dimension as $N$ by \ref{thm: modgr has almost split sequences}.
\end{proof}

We say that $V, W \in \Theta$ belong to the same column if there are almost split sequences $\xi_1, \ldots, \xi_l$ such that $V$ is a non-projective direct summand of the middle term of $\xi_1$, $W$ is a non-projective direct summand of the middle term of $\xi_l$ and for each $1 \leq i \leq l-1$, the middle terms of $\xi_i$ and $\xi_{i+1}$ have a common non-projective direct summand. This defines an equivalence relation on $\Theta$. We call the equivalence classes the columns of $\Theta$.

\begin{proposition}
Suppose there is a a duality $(-)^o: \modgr \rightarrow \modgr$. If $\Theta$ contains a module $V$ such that $V^o \cong V$, then $N^o \cong N$ for all $N$ in the column of $V$. The column of $V$ is a symmetry axis with respect to $(-)^o$.
\label{prop: vertical line through self-dual module is a symmetry axis}
\end{proposition}

\begin{proof}
As $(-)^o$ is a duality, $\Theta^o$ is a component of $\Gamma(\modgr)$ and $\Theta \cap \Theta^o = \emptyset$ or $\Theta = \Theta^o$. Thus, $V \in \Theta \cap \Theta^o$ implies $\Theta = \Theta^o$. 
Since $(-)^o$ does not change the quasi-length of modules, any almost split sequence in $\Theta$ such that the middle term has a self-dual summand is mapped to itself by $(-)^o$ and every direct summand of a middle term of such a sequence is self-dual. 
\end{proof}

\section{Representations of infinitesimal Schur algebras}
\label{section: Almost split sequences for infinitesimal Schur algebras}

In this section, we will introduce our main objects of study and show how they fit into the framework of the previous section. For algebraic groups, we follow the notation and terminology of \cite{Jantz1}. For an introduction to algebraic monoids, we refer the reader to \cite{Ren1}.\\
Let $k$ be an algebraically closed field of characteristic $p > 0$ and $\Mat_n$ be the monoid scheme of $(n \times n)$-matrices over $k$. Let $d \in \N_0$ and denote by $A(n, d)$ the space generated by all homogeneous polynomials of degree $d$ in the coordinate ring $k[\Mat_n] = k[X_{ij}\mid 1 \leq i,j \leq n]$ of $\Mat_n$. Then $k[\Mat_n] = \bigoplus_{d \geq 0} A(n, d)$ is a graded $k$-bialgebra with comultiplication $\Delta: k[\Mat_n] \rightarrow k[\Mat_n] \otimes_k k[\Mat_n]$ and counit $\epsilon: k[\Mat_n] \rightarrow k$ given by $\Delta(X_{ij}) = \sum_{l=1}^n X_{il} \otimes_k X_{lj}$, $\epsilon(X_{ij}) = \delta_{ij}$ and each $A(n, d)$ is a finite-dimensional subcoalgebra. As $\GL_n$ is dense in $\Mat_n$, the canonical map $k[\Mat_n] \rightarrow k[\GL_n]$ is injective, so that $k[\Mat_n]$ and each $A(n, d)$ can be viewed as a subcoalgebra of $k[\GL_n]$. Now let $G$ be a closed subgroup scheme of $\GL_n$ and $\pi: k[\GL_n] \rightarrow k[G]$ the canonical projection. Set $A(G) = \pi(k[\Mat_n])$ and $A_d(G)= \pi(A(n, d))$ as well as $ S_d(G) = A_d(G)^*$. Since $\pi$ is a homomorphism of Hopf algebras, $A(G)$ is a subbialgebra of $k[G]$ and each $A_d(G)$ is a subcoalgebra of $A(G)$. Thus, $S_d(G)$ is a finite-dimensional associative algebra in a natural way.   Following \cite{Doty1}, we say that a rational $G$-module $V$ is a polynomial $G$-module if the corresponding comodule map $V \rightarrow V \otimes_k k[G]$ factors through $V \otimes_k A(G)$. If the comodule map factors through $V \otimes_k A_d(G)$ for some $d \in \N_0$, we say that $V$ is homogeneous of degree $d$. Clearly, every $A_d(G)$-comodule is an $A(G)$-comodule and every $A(G)$-comodule is a $G$-module in a natural way.
As $A(G)$ is a factor bialgebra of $k[\Mat_n]$, it corresponds to a closed submonoid $M$ of $\Mat_n$. Since $A(G) \subseteq k[G]$ is a subbialgebra, $G \subseteq M$ is a dense subscheme, so that $M$ is the closure of $G$ in $\Mat_n$. Hence polynomial representations of $G$ can be regarded as rational representations of the algebraic monoid scheme $M = \overline{G}$. Note that all these notions depend on the given embedding $G \rightarrow \GL_n$.\\
We are mainly interested in the algebras $S_d(G_r T)$ and $S_d(B_r T)$, where $G = \GL_n$, $T \subseteq G$ is the maximal torus of diagonal matrices and $B \subseteq G$ is the Borel subgroup of upper triangular matrices. The algebra $S_d(G_r T)$ is the infinitesimal Schur algebra introduced in \cite{DNP2} and $S_d(B_r T)$ can be viewed as an infinitesimal analogue of the Borel-Schur algebra $S_d(B)$ introduced in \cite{G1}. Note that all results about $B_r T$ in this paper can also be proved for the Borel subgroup of lower triangular matrices instead. 
If $D = \overline{T} \subseteq \Mat_n$ is the monoid scheme of diagonal matrices, a comparison of coordinate rings shows $\overline{G_r T} = M_r D$, where $M_r D = (F_{\Mat_n}^r)^{-1}(D)$ with $F_{\Mat_n}^r$ being the $r$-th iteration of the Frobenius homomorphism on $\Mat_n$, see \cite{DNP2}. By the same token, we get for $L = \overline{B} \subseteq \Mat_n$ the monoid scheme of lower triangular matrices that $\overline{B_r T} = L_r D$, where $L_r D = (F_{L}^r)^{-1}(D)$ with $F_{L}^r$ the restriction of $F_{\Mat_n}^r$ to $L$. We have $k[M_r D] = k[X_{ij} \mid 1 \leq i, j \leq n]/(X_{ij}^{p^r} \mid  i \neq j)$ and $k[L_r D] = k[X_{ij} \mid 1 \leq i \leq j \leq n]/(X_{ij}^{p^r} \mid i \neq j)$.\\
From now on, we fix an embedding $G \subseteq \GL_n$ of a smooth connected algebraic group scheme and let $T \subseteq G$  be a maximal torus. A routine verification shows
\begin{lemma}
The category of polynomial $G_r T$-modules is closed with respect to submodules, factor modules and finite direct sums.

\label{lem: subcategory of polynomial GrT-modules closed w. r. to operations}
\end{lemma}

For $\T = \m \overline{G_r T}$, we let $\FG = t$, $\GG = u$ be the functors defined above \ref{prop: M t-acyclic, I min res => t(I) min res}. \\
If the center of $\GL_n$ is contained in $G$, we can decompose every polynomial $G_r T$-module into homogeneous constituents.

\begin{proposition} 
Suppose that $Z(\GL_n) \subseteq G$.
Let $V \in \m \overline{G_r T}$. Then there is a decomposition $V = \bigoplus_{d \in \N_0} V_d$ such that $V_d$ is an $S_d(G_r T)$-comodule for all $d \in \N_0$. 
\label{prop: decomposition of MrD-modules into homogeneous constituents}
\end{proposition}

\begin{proof}
This follows from \cite[1.3, 1.5]{Doty1}.
\end{proof}
From now on, we always assume $Z(\GL_n) \subseteq G$. 
As any $S_d(G_r T)$-comodule is also a $\overline{G_r T}$-comodule, we get

\begin{corollary}
If $V \in \m \overline{G_r T}$ is indecomposable, then $V = V_d$ for some $d \in \N_0$.
\label{cor: indecomposable MrD-modules are homogeneous}
\end{corollary}

The two preceding results show that the blocks of $\m \overline{G_r T}$ are just the blocks of the $S_d (G_r T)$ for $d \in \N_0$. Thus, injective and projective indecomposables as well as almost split sequences in $\m S_d(G_r T)$ coincide with those in $\m \overline{G_r T}$. In particular, $\m \overline{G_r T}$ has almost split sequences. \\
As $Z(\GL_n) \subseteq G$ is a torus contained in the center of $G_r T$, the $Z(\GL_n)$-weight spaces of a $G_r T$-module $V$ are $G_r T$-submodules of $V$. This shows that $Z(\GL_n)$ acts via a single character on indecomposable $G_r T$-modules $V$ and that indecomposable modules affording different characters belong to different blocks. If $V$ is a module for $\overline{G_r T}$, these submodules are just the homogeneous components of \ref{prop: decomposition of MrD-modules into homogeneous constituents}. We say that $V \in \m G_r T$ is homogeneous of degree $d$ if $Z(\GL_n)$ acts on $V$ via the character given by $d \in \Z$. If $V = V_d, W = W_{d'} \in \m G_r T$ are homogeneous such that $d \neq d'$, then any exact sequence starting in $V$ and ending in $W$ splits as the middle term is the direct sum of its homogeneous constituents. As almost split sequences are non-split, we get

\begin{proposition}
Let $\Theta$ be a component of the stable Auslander-Reiten quiver of $\m G_r T$. Then there is $d \in \Z$ such that all modules in $\Theta$ are homogeneous of degree $d$. 
\end{proposition}

For $G \neq \GL_n$ with maximal torus $T \subseteq G$, it is not known whether a $G_r T$-module $V$ such that all $T$-weights of $V$ are polynomial is a polynomial $G_r T$-module. However, for $G = \GL_n$, this was proved by Jantzen in \cite[Appendix]{Nak1}. We can also prove it for $B_r T$, where $B \subseteq \GL_n$ is the Borel subgroup of lower triangular matrices.
In the remainder of this section, let $G = \GL_n, T \subseteq \GL_n$ the torus of diagonal matrices, $B \subseteq \GL_n$ the Borel subgroup of lower triangular matrices, $U \subseteq B$ the subgroup of unipotent lower triangular matrices and $X \in \{G_r T, B_r T\}$, $D \subseteq \mathrm{Mat}_n$ the monoid scheme of diagonal matrices and $X(D)$ the character monoid of $D$. We identify $X(D)$ with a submonoid of $X(T)$ and call this submonoid the set of polynomial weights.
\begin{theorem}
If $V$ is an $X$-module such that all $T$-weights of $V$ are polynomial, then $V$ lifts to $\overline{X}$.
\label{thm: X-module lifts iff all weights are polynomial}
\end{theorem}

\begin{proof}
For $X = G_r T$, this is Jantzens result in \cite[Appendix]{Nak1}. Let $X = B_r T = U_r \rtimes T$, $\rho: V \rightarrow V \otimes_l k[B_r T]$ be the comodule map and $v \in V$. Since all $T$-weights of $V$ are polynomial, $V \vert_T$ lifts to $D$, so there are $f_i \in k[D], v_i \in V$ such that 
\begin{equation*}
t. (v \otimes 1) = \sum_{i=1}^s v_i \otimes f_i(t) 
\end{equation*}
for all $t \in T(A)$ and all commutative $k$-algebras $A$. Since $B_r T = U_r T = U_r \rtimes T$, there are $g_{ij} \in k[U_r], v_{ij} \in V$ such that 
\begin{equation*}
(ut).(v\otimes 1) = u.(t.(v\otimes 1)) = \sum_{i=1}^s \sum_{j=1}^l v_{ij} \otimes g_{ij}(u)f_i(t)
\end{equation*}
for all $t \in T(A), u \in U_r(A)$ and all commutative $k$-algebras $A$. Applying this to $A = k[B_r T]$ and $\id_A \in B_r T(A)$, we see that
\begin{equation*}
\rho(v)= \sum_{i=1}^s \sum_{j=1}^l v_{ij} \otimes g_{ij} \otimes f_i \in V \otimes_k k[U_r] \otimes_k k[D] = V \otimes_k k[L_r D],
\end{equation*}
so that $V$ lifts to $L_r D$.
\end{proof}

Viewing $X$-modules as $\Z^n$-graded $G_r$ resp. $B_r$-modules, the theorem shows that $V \in \m \overline{X}$ if and only if $V \in \m X$ and $\supp(V)\subseteq \N_0^n$.

\begin{corollary}
The category $\m \overline{X}$ is extension-closed in $\m X$.
\end{corollary}

Thus, $\m \overline{X}$ is a torsion class and a torsion-free class, so that all results of the previous section can be applied to $\m \overline{X}$ and we have $t = \FX, u = \GX$ for the torsion radical $t$ and the functor $u(V) = V / t(V)$. 

The transposition map $G_r T \rightarrow G_r T$ is an anti-automorphism inducing a duality $(-)^o: \m G_r T \rightarrow \m G_r T$ which restricts to a duality of $\m M_r D$ and $\m S_d(G_r T)$ for every $d \in \N_0$, see \cite[2.1]{DNP1}. Letting $B^- \subseteq \GL_n$ be the Borel subgroup of lower triangular matrices and defining $L_r^- D$ as above, transposition is an anti-isomorhpism $B_r T \rightarrow B_r^- T$ inducing anti-equivalences $\m B_r T \rightarrow \m B_r^- T$ and $\m L_r D \rightarrow \m L_r^- D$. Following \cite{DNP1}, we call these dualities \textbf{contravariant duality}. \\

\begin{remark}
Analogues of this duality can be constructed for some other reductive subgroups $G \subseteq \GL_n$ by using \cite[5.2]{Ren1} to extend the anti-automorphism of \cite[II.1.16]{Jantz1} to the closure of $G$ in $\mathrm{Mat}_n$. 
\end{remark}
\\ \\
Applying the canonical isomorphism  $(V^o)^o \cong V$ to the submodule $(\GX(V^o))^o$ of $(V^o)^o$, we get the following connection between $\FX$ and $\GX$.

\begin{proposition}
There is an isomorphism $\FX(V) \cong (\GX(V^{o}))^{o}$ natural in $V \in \m X$, where we abuse notation for $X = B_r T$ by letting $\GX(V^o)$ denote the largest $L_r^- D$-factor module of $V^o$. 
\label{prop: F(M) = G(M^o)^o}
\end{proposition}
In the remainder of this section, let $\Theta$ be a regular component of type $\ZAinf$ of the stable Auslander-Reiten quiver $\Gamma_s(G_r T)$ of $\ G_r T$.
By \cite[II.9.6(13)]{Jantz1}, we have $\widehat{L_r}(\lambda)^o \cong \widehat{L_r}(\lambda)$ for all $\lambda \in X(T)$ and the simple module $\widehat{L_r}(\lambda)$. We can use this to determine the position of indecomposable $\Ext$-projective modules $V \in \m M_r D$ in $\Theta$.

\begin{proposition}
Let $V \in \Theta \cap \m M_r D$ be $\Ext$-projective in $\m M_r D$. Then $V$ is quasi-simple and lies at the left side of a wing in $\Theta \cap \m M_r D$.

\label{prop: projectives in mod MrD which are in ZAinf-components in mod GrT are quasi-simple and on the left of a mountain}
\end{proposition}

\begin{proof}
This follows directly from \ref{lem: Ext-projective in T for graded algebras, almost split sequences in T}, \ref{cor: projective $T$-modules in ZAinfty components of modgr with duality fixing the simples} and \ref{lem: mountains in ZAinfty}.
\end{proof}

Since $\m G_r T$ is a Frobenius category, there are no non-projective $G_r T$-modules of finite projective dimension. However, there are $M_r D$-modules with finite projective dimension.

\begin{corollary}
Let $V \in \Theta \cap \m M_r D$ and $P$ be the projective cover of $V$ in $\m G_r T$. Suppose $P \in \m M_r D$ and $pd_{M_r D}(V)=1$.  Then $V$ is quasi-simple and $\tau_{G_r T}(V) \notin \m M_r D$.
\end{corollary}

\begin{proof}
We have $\tau_{G_r T}(V) = \Omega^2_{G_r T}(V)$ by \cite[7.2.3]{Farn2}. Since $pd_{M_r D}(V)= 1$ and $P \in M_r D$, we have $\Omega_{G_r T}^1(V) = \Omega^1_{M_r D}(V)$ projective in $\m M_r D$. In particular, the projective cover of $\Omega_{G_r T}(V)$ in $\m G_r T$ is not an $M_r D$-module, so that $\Omega^2_{G_r T}(V)\notin \m M_r D$ by \ref{thm: X-module lifts iff all weights are polynomial}. As $\Omega^1_{G_r T}(\Theta)$ is a component isomorphic to $\Theta$ by \cite[p. 338]{ARS1}, \ref{prop: projectives in mod MrD which are in ZAinf-components in mod GrT are quasi-simple and on the left of a mountain} implies that $\Omega_{G_r T}^1(V)$ is quasi-simple. Hence $V$ is also quasi-simple.
\end{proof}

\begin{proposition} 
Suppose $V \in \m M_r D$ is indecomposable and $\G$-acyclic such that $pd_{M_r D}(V) < \infty$. Then there are only finitely many $M_r D$-modules in $\lbrace \tau_{G_r T}^i(V) \mid i \in \N_0 \rbrace$ and they all have finite projective dimension. 
\label{prop: M G-acyclic with finite projective dimension => only finitely many MrD-modules in positive tau-orbit}
\end{proposition}

\begin{proof}
By \cite[7.2.3]{Farn2}, we have $\tau_{G_r T} \cong \Omega^2_{G_r T}$. Setting $s = pd_{M_r D}(V)$, a dual version of \ref{prop: M t-acyclic, I min res => t(I) min res} shows that $\G(\tau_{G_r T}^l(V)) = \G(\Omega^{2l}_{G_r T}(V))= 0$ for $l > \frac{s}{2}$, so that $\tau^l_{G_r T}(V) \notin \m M_r D$. Since minimal projective resolutions for $V$ in $\m G_r T$ induce minimal projective resolutions for $\G(V)$ in $\m M_r D$, this also shows that all $M_r D$-modules in $\lbrace \tau_{G_r T}^i(V) \mid i \in \N_0 \rbrace$ have finite projective dimension.   
\end{proof}

\begin{corollary}
Suppose $\Theta$ contains a $\G$-acyclic $M_r D$-module $V$ which is quasi-simple such that $pd_{M_r D}(V)< \infty$  and a self-dual module $N$. Then $\Theta$ contains only finitely many $M_r D$-modules.
\end{corollary}

\begin{proof}
Since $X \in \m M_r D$ iff $X^o \in \m M_r D$ for all $X \in \m G_r T$, this follows from \ref{prop: M G-acyclic with finite projective dimension => only finitely many MrD-modules in positive tau-orbit} and \ref{prop: vertical line through self-dual module is a symmetry axis}. 
\end{proof}

\begin{proposition}
Let $V \in \m M_r D$  be $\G$-acyclic and indecomposable such that $pd_{M_r D}(V) \leq 2n$ and $\tau_{G_r T}^{n}(V) \in \m M_r D$. Then $\tau_{G_r T}^{n}(V)$ is projective in $\m M_r D$. If moreover $V \in \Theta$, then $V$ is quasi-simple. 
\end{proposition}

\begin{proof}
Let $\ldots \rightarrow P_1 \rightarrow P_0 \rightarrow V \rightarrow 0$ be a minimal projective resolution in $\m G_r T$. As $V$ is $\G$-acyclic,  a dual version of \ref{prop: M t-acyclic, I min res => t(I) min res} yields a minimal projective resolution $\ldots \rightarrow \G(P_1) \rightarrow \G(P_0) \rightarrow V \rightarrow 0$ in $\m M_r D$. Since $pd_{M_r D}(V) \leq 2n$, we get $\G(P_{2n+1})=0$. Now $\F(\tau_{G_r T}^{n+1}(V)) \cong \F(\Omega^{2n+2}_{G_r T}(V)) \subseteq \F(P_{2n+1}) \cong \G(P_{2n+1})^{o} = 0$ by \ref{prop: F(M) = G(M^o)^o} since $P_{2n+1}^{o} \cong P_{2n+1}$. By \ref{lem: Ext-projective in T for graded algebras, almost split sequences in T}, $\tau_{G_r T}^{n}(V)$ is projective in $\m M_r D$. If $V \in \Theta$, \ref{prop: projectives in mod MrD which are in ZAinf-components in mod GrT are quasi-simple and on the left of a mountain} shows that $V$ is quasi-simple.
\end{proof}

In the next result, we adapt a standard Morita-equivalence between blocks of $\m G_r T$ to blocks of $\m M_r D$. Later we will also need that this equivalence is compatible with the restriction functor to $(\SL_n)_r$. We denote by $St_s \cong \widehat{Z}_s((p^s - 1)\rho)$ the $s$-th Steinberg module, where $\rho$ is the half-sum of positive roots and we choose the root system of $\GL_n$ as in \cite[I.1.21]{Jantz1}. If $V \in \m G_r T$, we denote by $V^{[s]}$ the module obtained by composing the module structure on $V$ with the $s$-th iteration of the Frobenius morphism and by $k_{1/2 (p^s - 1)(n - 1) \det \vert_{G_r T}}$ the 1-dimensional module given by the character $\frac{1}{2} (p^s - 1)(n - 1) \det \vert_{G_r T}$ of $G_r T$.
\begin{proposition}
Suppose $p \geq 3$. Let $b$ be a block of $M_{r-s}D$. The functor $A: \m G_{r-s} T \rightarrow \m G_r T, V \mapsto St_s \otimes_k V^{[s]}\otimes_k k_{1/2 (p^s - 1)(n - 1) \det \vert_{G_r T}}$ commutes with the forgetful functors $\m G_{r - s}T \rightarrow \m (\SL_n)_{r-s}$, $\m G_r T \rightarrow \m (\SL_n)_r$ and induces a Morita-equivalence from $b$ to a block of $M_r D$.
\label{prop: equivalence between blocks of M_r-s D and M_r D}
\end{proposition}

\begin{proof}
We have $\mathrm{Ext}^{1}_{M_r D}(V_1, V_2) \cong \mathrm{Ext}^{1}_{G_r T}(V_1, V_2)$ for all $V_1, V_2 \in \m M_r D$ by \ref{thm: X-module lifts iff all weights are polynomial}, so there is a block $b'$ of $\m G_{r-s}T$ containing $b$.  By \cite[II.10.5]{Jantz1}, $V \mapsto St_s \otimes_k V^{[s]}$ induces a Morita-equivalence between $b'$ and a block of $G_r T$. Since shifting with a character of $G_r T$ is an equivalence of categories, $A$ also induces such an equivalence. We now show that $V \in \m M_{r-s}D$ iff $A(V) \in \m M_r D$. Since $St_s \cong \widehat{Z}_s((p^s - 1)\rho)$, so that the set of weights of $St_s$ is $W \cong S_n$-invariant by \cite[II.9.16(1)]{Jantz1} and $-(n-1)$ occurs as a coordinate of $2 \rho$ , we have that for every $i \in \lbrace 1, \ldots, n \rbrace$, there is a weight $\lambda$ of $St_s$ such that $\lambda_i = - \frac{1}{2}(p^s -1)(n-1)$. Suppose there is a weight $\mu$ of $St_s$ with a smaller coordinate. Using the action of the Weyl group, we may assume that $\mu_n < - \frac{1}{2}(p^s -1)(n-1)$. Then \cite[II.9.2(6)]{Jantz1} implies that $(p^s - 1)\rho - \mu$ is a sum of positive roots. As $((p^s - 1)\rho - \mu)_n > 0$ and there is no positive root with a positive $n$-th coordinate, this is a contradiction. Thus, there is no weight $\mu$ of $St_s$ with $\mu_ i < - \frac{1}{2}(p^s -1)(n-1)$. As the grading on $V^{[s]}$ arises by multiplying all degrees in the grading of $V$ by $p^s$, we get that $V$ has a weight with a negative coordinate iff $A(V)$ has a weight with a negative coordinate by the definition of the grading on a tensor product. Thus, $V \in \m M_{r-s} D$ iff $A(V) \in \m M_r D$. Hence $A$ maps the class of simple $M_{r-s}D$-modules of the $G_{r-s} T$-block containing $b$ to the class of simple $M_r D$-modules in the block of $G_r T$ containing $A(b)$. As $\mathrm{Ext}^{1}_{M_r D}(V_1, V_2) \cong \mathrm{Ext}^{1}_{G_r T}(V_1, V_2)$ for all $V_1, V_2 \in \m M_r D$, we see that $V_1, V_2$ belong to the same block of $M_{r-s}D$ iff $A(V_1), A(V_2)$ belong to the same block of $M_r D$, so that $A$ induces a Morita-equivalence from $b$ to a block of $M_r D$. Since $\det$ vanishes on $\SL_n$, the functor $A$ commutes with the forgetful functors.  
\end{proof}

\section{Modules of complexity one}
In this section and the following sections, let $p = \mathrm{char} (k) \geq 3$.
Let $G \subseteq \GL_n$ be a reductive algebraic group over $k$ with $Z(\GL_n) \subseteq G$, $T \subseteq G$ a maximal torus, $R$ the root system of $G$ relative to $T$ and $B = U \rtimes T \subseteq G$ a Borel subgroup containing $T$ with unipotent radical $U$. Let $R^+ \subseteq R$ be the set of positive roots determined by $B$. Recall that the complexity $cx(V)$ of a module $V$ is the polynomial rate of growth of a minimal projective resolution of $V$. Let $X \in \{G_r T, B_r T\}$ and $\Theta$ be a component of $\Gamma_s(X)$. By \cite[Section 1]{Farn3} and \cite[8.1]{Farn2}, the complexity of a module coincides with the dimension of its rank variety and we have $cx_{G_r T}(V_1) = cx_{G_r T}(V_2)$ for all $V_1, V_2 \in \Theta$, so that we may define $cx_{G_r T}(\Theta) = cx_{G_r T}(V_1)$. As in the previous section, we denote by $\overline{X}$ the closure of $X$ in $\Mat_n$. In this section, we will show that if $cx_{G_r T}(\Theta) = 1$, then $\Theta \cap \m \overline{X}$ is finite.\\  

\begin{lemma}
Let $X \in \{B_r T, G_r T\}$ and $cx_{X}(\Theta) = 1$.
\begin{enumerate}[(1)] 
\item If $X = G_r T$, there are $\alpha \in R$, $0 \leq s \leq r - 1$ such that $\tau_{X}^{p^s} (V) = V[p^r \alpha]$ for all $V \in \Theta$.
\item If $X = B_r T$, there are $\alpha \in R^+ \cup \{0\}, 0 \leq s \leq r-1$ such that $\tau_{X}^{p^s}(V) = V [p^r \alpha - 2p^s(p^r-1) \rho]$ for all $V \in \Theta$, where $\rho = \frac{1}{2}\sum_{\beta \in R^+} \beta$ is the half-sum of positive roots.
\end{enumerate}
\label{lem: tau for modules of complexity 1}
\end{lemma}

\begin{proof}
\begin{enumerate}[(1)]
\item This follows directly from \cite[6.1.2 , 7.2.3, 8.1.2]{Farn2}, noting that the proof of  \cite[8.1.2]{Farn2} only depends on the support variety and hence on the component of $V$.
\item Let $X = B_r T$. Then \cite[7.2.2]{Farn2} yields that $\mathcal{N}(V) = V \otimes_k k_{\lambda_B \vert_{B_r T}}[- p^r \lambda_B \vert_T]$ for the Nakayama functor $\mathcal{N}$ of $\m B_r T$, where $\lambda_B$ is the character of $B$ given by $\lambda_B = \det \circ \mathrm{Ad}$. By the remarks preceding \cite[7.2.2]{Farn2}, $\lambda_B \vert_T = 2 \rho$. As $B_r T \cong U_r \rtimes T$ with $U_r$ unipotent, tensoring with $k_{\lambda_B \vert_{B_r T}}$ only changes the $T$-action on a module, so that $V \otimes_k k_{\lambda_B \vert_{B_r T}} \cong V [2 \rho]$ and $\mathcal{N}(V)= V [- 2(p^r - 1) \rho]$. By \cite[6.1.2, 8.1.1]{Farn2}, there exist $\alpha \in R^+ \cup \{0 \}$ and $s \in \N_0$ such that $\Omega_{B_r T}^{2p^s} (V) \cong V \otimes_k p^r \alpha$, so that $\tau^{p^s}(V) \cong \mathcal{N}^{p^s}(V[p^r \alpha]) \cong V[p^r \alpha - 2p^s(p^r-1)\rho] $.

\end{enumerate}
\end{proof}

\begin{lemma}
Let $X \in \{B_r T, G_r T\}$, $V$ be an indecomposable $X$-module such that $cx_{X}(V)=1$. Then there are only finitely many polynomial $X$-modules in the $\tau_{X}$-orbit of $V$. 
\label{lem: cx(V)= 1 => tau - orbit contains only finitely many polynomial modules} 
\end{lemma}

\begin{proof}
Let $X = B_r T$ and $R^+ \subseteq R$ be the set of roots of $B$. By \ref{lem: tau for modules of complexity 1}, (2),  we have $\tau^{p^s}_{B_r T}(V) = V[p^r \alpha - 2p^s(p^r-1) \rho]$ for some $\alpha \in R^+ \cup \{0\}, 0 \leq s \leq r-1$. It follows that
\begin{equation*}
\supp(\tau_{B_r T}^{ip^s}(V)) = \supp(V) + ip^r \alpha - 2i p^s (p^r-1) \rho
\end{equation*}
for all $i \in \Z$. Since $p^r \alpha \neq 2p^s(p^r - 1)\rho$, we have $\supp(\tau_{B_r T}^{ip^s}(V)) \neq \supp(\tau_{B_r T}^{jp^s}(V))$ for $i \neq j$. As $Z(\GL_n) \subseteq G$, the degree $d$ weights of a maximal torus of $\GL_n$ containing $T$ surject onto the degree $d$ weights of $T$. As all weights in the sets $\supp(\tau_{B_r T}^{ip^s}(V))$ have degree $d$ and there are only finitely many polynomial weights of a given degree $d$ for $\GL_n$, only finitely many of these sets consist entirely of polynomial weights. Thus, the $\tau_{B_r T}^{p^s}$-orbit of $M$ contains only finitely many polynomial $B_r T$-modules. As the $\tau_{B_r T}$-orbit of $V$ is the union of the $\tau_{B_r T}^{p^s}$-orbits of $V, \tau_{B_r T}(V), \ldots, \tau_{B_r T}^{p^s-1}(V)$, the result follows.\\
For $X = G_r T$, apply \ref{lem: tau for modules of complexity 1}, (1), and use analogous arguments. 
\end{proof}

We denote by $F$ the forgetful functor $\m G_r T \rightarrow \m G_r$ resp. $\m B_r T \rightarrow \m U_r$.

\begin{proposition}
Let $X \in \{B_r T, G_r T\}$ and $V$ be an indecomposable polynomial $X$-module such that $cx_{X}(V)=1$ and $\Theta$ be the component of $\Gamma_s(X)$ containing $V$. Then $\Theta$ contains only finitely many polynomial $X$-modules. 
\label{prop: only finitely many polynomial modules in component of cx 1}
\end{proposition}

\begin{proof}
We first show that $\Theta$ is either regular or has only finitely many $\tau_{X}$-orbits.\\
Let $F$ be the forgetful functor. If $X = G_r T$, then $\Theta$ is regular; else there would be a projective indecomposable $G_r T$-module $P$ such that $\Rad(P) \in \Theta$ by \cite[V.5.5]{ARS1}, so that $cx_{G_r T}(\Rad(P)) = 1$ and the simple $G_r T$-module $S = \Omega^{-1}_{G_r T}(\Rad(P))$ would also have complexity $1$. But then $cx_{G_r}(F(S)) = 1$, a contradiction to \cite[Lemma 2.2]{FarnRoehr1}.
If $X = B_r T = U_r \rtimes T$ and $\Theta$ is a non-regular component of complexity $1$, the above arguments show that there is a simple $B_r T$-module $k_{\lambda}$ of complexity $1$, so that $cx_{U_r}(F(k_{\lambda})) = cx_{U_r}(k)=1$. As the rank variety of $(U_r)_2$ is contained in that of $U_r$, we get   $cx_{(U_r)_2} \leq 1$ and \cite[Theorem 2.7]{FarnVoigt1} shows that $U_r$ has finite representation type. Thus, $F(\Theta)$ is finite, and \cite[5.6]{Farn3} implies $F(\Theta) \cong \Z[A_{p^l}]/(\tau)$ for some $l \in \N$, so that there are two distinct vertices in $F(\Theta)$ with only one successor and that all other vertices in $F(\Theta)$ have exactly two predecessors. Since $N \in \m B_r T$ is projective iff $F(N) \in \m U_r$ is projective by \ref{prop: indecomposables, projectives in modgr}, it follows from \ref{thm: modgr has almost split sequences} that there are two distinct $\tau_{B_r T}$-orbits in $\Theta$ such that every element of these orbits has exactly one predecessor and all other orbits consist of vertices with exactly two predecessors. Thus, $\Theta$ does not have tree class $A_{\infty}, A_{\infty}^{\infty}$ or $D_{\infty}$. Then \cite[3.4]{FarnRoehr1} (which does not depend on the group being reductive) implies that $\Theta$ has only finitely many $\tau_{B_r T}$-orbits.\\
Thus, $\Theta$ is either regular or has only finitely many $\tau_{X}$-orbits.  
In the second case, the result follows from \ref{lem: cx(V)= 1 => tau - orbit contains only finitely many polynomial modules}. In the first case, \cite[5.6]{Farn3} and \cite[8.2.2]{Farn2} imply $\Theta \cong \ZAinf$. 
Letting $\Theta \cong \ZAinf$ and $N \in \Theta$ be quasi-simple, the set of quasi-simple modules in $\Theta$ is the $\tau_X$-orbit of $N$. Thus, \ref{lem: cx(V)= 1 => tau - orbit contains only finitely many polynomial modules} shows that there are only finitely many quasi-simple polynomial $X$-modules in $\Theta$, so that the result follows from \ref{lem: only finitely many quasi-simple modules in theta cap mod MrD => theta cap mod MrD finite}.
\end{proof}

In the remainder of this section, let $G = \GL_n$ with the notation of Section \ref{section: Almost split sequences for infinitesimal Schur algebras}, $T \subseteq G$ the maximal torus of diagonal matrices and $B \subseteq G$ the Borel subgroup of upper triangular matrices.

\begin{proposition}
Let $F(\Theta) = \ZAinf / \langle\tau^{p^t}\rangle$ and $V \in \Theta \cap \m \overline{X}$ such that $ql(V)> \frac{3p^t-2}{2}$. Then the following statements hold:
\begin{enumerate}[(1)]
\item  For every $N \in \Theta \cap \m \overline{X}$ such that $ql(N)> \frac{p^t}{2}$, there is an undirected path from $V$ to $N$ in $\Theta \cap \m \overline{X}$.
\item $\Theta \cap \m \overline{X}$ contains a unique module of maximal quasi-length.
\end{enumerate}

\label{prop: if ql(M) large in component of type ZAinf/tau^p^s, path to other N with gl(N) large in mod MrD}  
\end{proposition}

\begin{proof}

We first show: 
\begin{center}
$(\ast)$ If $\tau^{ip^t}_{X}(N) \notin \m \overline{X}$, then $\tau^{jp^t}_{X}(N) \notin \m \overline{X}$ for all $j > i$.
\end{center}
Let $X = B_r T$. Since $F(\Theta) \cong  \ZAinf / \langle\tau^{p^t}\rangle$, we get that for each $N \in \Theta$, there is $\lambda_N \in X(T)$ such that $\tau_X^{p^t}(N) = N[\lambda]$. As shifting with $\lambda$ maps almost split sequences to almost sequences, we get $\lambda_N = \lambda_{N'}$ for every $N'$ which is a successor, predecessor or element of the $\tau_X$-orbit of $N$. Since $\Theta$ is a connected component, we inductively get that $\lambda_N = \lambda_{N'}$ for all $N, N' \in \Theta$. Thus, $(\tau_X \vert_{\Theta})^{p^t} = [\lambda]$ for some $\lambda \in X(T)$. By \ref{lem: tau for modules of complexity 1}, (2), we also find suitable $\alpha \in R^+ \cup \{0\}, s \in \N_0$ such that $(\tau_X \vert_{\Theta})^{p^s} = [p^r \alpha - 2p^s(p^r-1)\rho]$ for all $V \in \Theta$. As $G = \GL_n$ and we have chosen the roots as in \cite[II.1.21]{Jantz1}, the first coordinate of $p^r \alpha - 2 p^s(p^r-1)\rho$ is negative and the last coordinate is positive, so that the same holds for $\lambda$. We get that if $N \in \m \overline{X} \cap \Theta$ and $\tau^{p^t}_{X}(N) \notin \m \overline{X}$, then $\tau^{jp^t}_{X}(N) \notin \m \overline{X}$ for all $j \in \N$, showing the statement for $B_r T$. For $X = G_r T$, use \ref{lem: tau for modules of complexity 1}, (1), and argue as above.\\
Analogously, one shows: If $\tau^{-ip^t}_{X}(N) \notin \m \overline{X}$, then $\tau^{-jp^t}_{X}(N) \notin \m \overline{X}$ for all $j > i$.
\begin{enumerate}[(1)]
\item Since every module on the downward sectional path starting in $N$ or $V$ is also an $\overline{X}$-module, we may assume $ql(N)= \frac{p^t+1}{2}$ and $ql(V)= \frac{3p^t-1}{2}$. By following the downward sectional path starting in $V$ and the upward sectional path from the quasi-socle of $V$ to $V$ for $p^t - 1$ steps, we get a module $N' \in \Theta \cap \m \overline{X}$ such that $ql(N')= ql(V) - (p^t - 1)= ql(N)$. Then the modules $\tau_{G_r T}^i (N')$ belong to the wing of $V$ for $1 \leq i \leq p^t - 1$, so that they belong to $\m \overline{X}$ by \ref{lem: mountains in ZAinfty}. Thus, there are $l \in \N_0$ and $0 \leq i \leq p^t - 1$ such that $N = \tau_{X}^{i + lp^{t}}(N')$.  By $(\ast)$, $\tau_{X}^{l p^t}(\tau_{X}^i (N'))\in \m \overline{X}$ implies $\tau_{X}^{k p^t}(\tau_{X}^i (N'))\in \m \overline{X}$ for $1 \leq k \leq l$. Since $ql(N) = \frac{p^t+1}{2}$, \ref{lem: M, tau^i M in mod MrD for ql(M) large enough} shows that $\tau^j_{X}(N') \in \m \overline{X}$ for $1 \leq j \leq i + lp^{t}$. As $\m \overline{X}$ is extension-closed, the middle terms of the almost split sequences defined by these modules also belong to $\m \overline{X}$, so that we get an undirected path in $\Theta \cap \m \overline{X}$ from $N'$ to $N$. As there is also such a path from $V$ to $N'$, the result follows. 
\item If $N, N' \in \Theta \cap \m \overline{X}$ have maximal quasi-length and $N \neq N'$, $(\ast)$ and \ref{lem: M, tau^i M in mod MrD for ql(M) large enough} imply the existence of a module with greater quasi-length in $\Theta \cap \m \overline{X}$, a contradiction.
\end{enumerate}
\end{proof}

\begin{corollary}
Let $X \in \{B_r T, G_r T\}$. Suppose that $F(\Theta)$ is a homogeneous tube and $\Theta \cap \m \overline{X} \neq \emptyset$. Then there is a unique module $N \in \Theta \cap \m \overline{X}$ such that $\Theta \cap \m \overline{X}$ consists of the wing of $N$.
\label{cor: mountains in components corresponding to homogeneous tubes}
\end{corollary}

\begin{proof}
Taking $t = 0$ in \ref{prop: if ql(M) large in component of type ZAinf/tau^p^s, path to other N with gl(N) large in mod MrD}, we see that $\Theta \cap \m \overline{X}$ is connected. Since $\m \overline{X}$ is extension-closed by \ref{thm: X-module lifts iff all weights are polynomial}, we get that $\Theta \cap \m \overline{X}$ is equal to the wing of a module $V \in \Theta \cap \m \overline{X}$ with maximal quasi-length.
\end{proof}
Note that for $X = G_r T$ and $r = 1$, $F(\Theta)$ always is a homogeneous tube by \ref{lem: tau for modules of complexity 1}.

\section{The case $n = 2, r = 1$}
In this section, let $G = \GL_2$, $T \subseteq G$ be the torus of diagonal matrices. In \cite{DNP1}, quiver and relations for the blocks of $S_d(G_1 T)$ as well as the number of blocks were determined and it was shown that all blocks in this case are representation-finite. Our aim in this section is to determine the Auslander-Reiten quiver for the algebras $S_d(G_1 T)$ by first considering the position of the relevant modules in the stable Auslander-Reiten quiver of $G_1 T$. Since $G_1 T = (\SL_2)_1 T$ and modules for $(\SL_2)_1$ correspond to modules for the restricted enveloping algebra $U_0(\mathfrak{sl}_2)$, we have a restriction functor $F: \m G_1 T \rightarrow \m U_0(\mathfrak{sl}_2)$ and the results of Section \ref{section: GrT - modules and graded algebras} apply to this functor. Contrary to the $\Z$-grading obtained from the usual theory of $(\SL_2)_1 (T\cap \SL_2)$-modules, we obtain a $\Z^2$-grading on $\Sl_2$ and $U_0(\Sl_2)$ given by $\deg(e)=(1, -1), \deg(f) = (-1, 1)$ and $\deg(h)= (0, 0)$ for the standard basis $e, f, h$ of $\Sl_2$. The indecomposable $U_0(\Sl_2)$-modules were classified by Premet in \cite{Pre1}. For $d \geq 0$, let $V(d)$ be the Weyl module of highest weight $d$ for the group scheme $\SL_2$. Then $V(d)$ has a basis $v_0,\ldots, v_d$ such that for the standard basis $\lbrace e, f, h \rbrace \subseteq \Sl_2$, we have 
\begin{equation*}
e.v_i = (i + 1)v_{i+1}, f.v_i = (d - i + 1)v_{i-1}, h.v_i = (2i - d) v_i
\label{eq: action of sl2 on vi}
\end{equation*}
For $d = sp + a$, where $s \geq 1$ and $0 \leq a \leq p-2$, the space 
\begin{equation*}
W(d):= \bigoplus_{i = a + 1}^{d} k v_i \subseteq V(d) 
\end{equation*}
is a maximal $U_0(\Sl_2)$-submodule of $V(d)$.
We record the classification of indecomposable $U_0(\Sl_2)$-modules in the following theorem:

\begin{theorem}
The following statements hold: \\
(1) Let $C \subseteq \SL_2$ be a complete set of coset representatives of $\SL_2 / B$, where $B \subseteq \SL_2$ is the Borel subgroup of upper triangular matrices. Then any nonprojective indecomposable $U_0(\Sl_2)$-module is isomorphic to exactly one of the modules of the following list:
\begin{enumerate}[(i)]
\item $V(d), V(d)^*$ for $d \geq p, d \not \equiv -1  \thinspace \mathrm{mod} \thinspace p$;
\item $V(r) =: L(r)$ for $0 \leq r \leq p-1$;
\item $g.W(d)$ for $g \in C$ and $d = sp+a$ with $s \geq 1$ and $ 0 \leq a \leq p - 2$.
\end{enumerate}
In particular, the modules appearing in the list are pairwise nonisomorphic. \\
(2)Up to isomorphism, every indecomposable $U_0(\Sl_2)$-module $N$ is uniquely determined by the triple $(\dim_k N, \soc_{\Sl_2}(N), V_{\Sl_2}(N))$, where $V_{\Sl_2}(N)$ is the $U_0(\Sl_2)$ - rank variety of $N$.
\label{theorem: classification of U_0(sl2)-modules}
\end{theorem}

The stable Auslander-Reiten quiver of $U_0(\Sl_2)$ is the disjoint union of $p - 1$ components of type $\Z[\tilde{A}_{12}]$ and infinitely many homogeneous tubes. The modules $g.W(d)$ lie in homogeneous tubes and $ql(g.W(d)) = s$, where $d = sp + a$ (\cite[4.1.2]{Farn1}). We first determine all indecomposable homogeneous submodules and factor modules of $V(sp + a)$. By definition (see for example \cite[Section 4.1]{Farn1}), the $\Sl_2$-structure on $V(sp + a)$ comes from a twist of the dual of the $d$-th symmetric power of the natural $\SL_2$-module $L(1) = k^2$. This module is a $\GL_2$-module in a natural way and the differential of the $\GL_2$-action restricts to the $\Sl_2$-action on this module. Thus, we get  an $X(T)$-grading given by $\deg (v_i) = (i, sp +a - i)$ compatible with the $\Sl_2$-action. Applying contravariant duality, we see that $V(sp+a)^o$ has a basis with the same weights. Letting $w_0 \in G$ be a representative of the nontrivial element of the Weyl group $W$ of $G$, the submodules $W(sp +a)$ resp. $w_0.W(sp+a)$ of $V(sp+a)$ have bases $v_{a+1},\ldots, v_{sp+a}$ resp. $v_0, \ldots, v_{sp - 1}$, so they are homogeneous. Other modules of the form $g.W(d)$ don't have $T$-invariant rank varieties and are thus not $G_1 T$-modules. We have $w_0.W(sp+a) = W(sp+a)^{w_0}$, the twist of $W(sp+a)$ by $w_0$. Since $w_0. (t_1, t_2) = (t_2, t_1)$ for all $t_1, t_2 \in T$, the grading on $w_0. V$ is given by interchanging the coordinates of the grading of $V$ for all $V \in \m G_1 T$. We denote the $G_1 T$-structures defined above on $V(sp+a), V(sp+a)^o, W(sp+a)$ and $W(sp+a)^{w_0}$ by $\widehat{V}(sp+a), \widehat{V}(sp+a)^o, \widehat{W}(sp+a)$ and $\widehat{W}(sp+a)^{w_0}$, respectively. For $0 \leq r \leq p-1$, we also write $\widehat{L}(r):= \widehat{V}(r)$. By the classification in \ref{theorem: classification of U_0(sl2)-modules}, we have $V(d) \cong V(d)^*  \cong V(d)^o$ for $0 \leq d \leq p-1$. This shows that $F(\widehat{V}(sp+a)^o) \cong V(sp+a)^*$.

\begin{lemma}
Let $V$ be a a nonprojective indecomposable $G_1 T$-module. Then there are $x, y  \in \Z$ and a basis $v_0, \ldots, v_l$ of $V$ such that $\deg(v_i)= (x + i, y + l - i)$. 
\label{lem: homogeneous submodules of V(sp + a) are spanned by vi,..., vi+k}
\end{lemma}

\begin{proof}
By \ref{theorem: classification of U_0(sl2)-modules} and its succeeding remarks, $F(V)$ is isomorphic to either of $V(sp+a) \cong F(\widehat{V}(sp+a)), V(sp+a)^* \cong F(\widehat{V}(sp+a)^o)$ with $s \geq 0$ or $W(sp+a) \cong F(\widehat{W}(sp+a))$ or $W(sp+a)^{w_0} \cong F(\widehat{W}(sp+a)^{w_0})$ with $s \geq 1$. Let $F(V) \cong V(sp+a) \cong F(\widehat{V}(sp+a))$. An application of \ref{prop: indecomposables, projectives in modgr} and \ref{prop: mod HT sum of blocks of mod H rtimes T} yields $\lambda \in X(T)$ such that $\lambda \vert _{T \cap (\SL_2)_1}=1$ and $V \cong \widehat{V}(sp+a)[\lambda]$. By the remarks above, $\widehat{V}(sp+a)$ has a basis $v_0, \ldots, v_{sp+a}$ such that $\deg(v_i) = (i, sp+a - i)$. In $\widehat{V}(sp+a)[\lambda]$, this basis has degrees $\deg(v_i) = (\lambda_1 + i, \lambda_2 +sp+a-i)$, so that $(x, y) = (\lambda_1, \lambda_2)$ yields the claim in this case. As $\widehat{V}(sp+a)^o$ has a basis with the same degrees, the preceding arguments also apply in the case $F(V) \cong F(\widehat{V}(sp+a)^o)$. In the remaining two cases, we apply the same arguments to the bases $v_{a+1}, \ldots, v_{sp+a}$ resp. $v_0, \ldots, v_{sp-1}$ of $\widehat{W}(sp+a)$ resp. $\widehat{W}(sp+a)^{w_0}$, setting $(x, y) = (\lambda_1 + a +1, \lambda_2)$ resp. $(x, y) = (\lambda_1, \lambda_2 + a + 1)$. 
\end{proof}

Using the lemma, we can now determine the $M_1 D$-parts of the $G_1 T$-components containing $\widehat{W}(sp+a), \widehat{W}(sp+a)^{w_0}, \widehat{V}(sp+a)$ and $\widehat{V}(sp+a)^o$.

\begin{proposition}
Let $\Theta$ be the component of the stable Auslander-Reiten quiver of $\m G_1 T$ containing the module $\widehat{W}(sp + a)$ resp. $\widehat{W}(sp+a)^{w_0}$. Then $\Theta$ is a $\ZAinf$-component such that $\tau_{G_1 T} \vert_{\Theta}=[( p, -p)]$ resp. $\tau_{G_1 T} \vert_{\Theta}=[( -p, p)]$  and $\Theta \cap M_1 D$ is the wing of $\widehat{W}(sp + a)$ resp. $\widehat{W}(sp+a)^{w_0}$. The $i$-th predecessor of $\widehat{W}(sp+a)$ resp. $\widehat{W}(sp+a)^{w_0}$ on the sectional path starting in its quasi-socle is $\widehat{W}((s-i)p+a)[(ip, 0)]$ resp. $\widehat{W}((s-i)p+a)^{w_0}[(0, ip)]$. The quasi-socle $\widehat{W}(p+a)[(s-1)p, 0)]$ resp. $\widehat{W}(p+a)^{w_0}[ (0, (s-1)p)]$ is projective in $\m M_1 D$. 
\label{prop: component of mod G_1 T containing W(sp + a)}
\end{proposition}

\begin{proof}
We show the result for $\widehat{W}(sp+a)$. For $\widehat{W}(sp+a)^{w_0}$, the result then follows since twisting with $w_0$ is an equivalence of categories interchanging the two coordinates of the grading. Since $F(\Theta)$ is a homogeneous tube, we have $cx_{G_1 T}(\Theta) = cx_{U_0(\Sl_2)}(F(\Theta)) = 1$, so that $\Theta$ is of type $\ZAinf$ by \cite[8.2.2]{Farn2}. We first consider the case $s > 1$. In $\m U_{0}(\Sl_2)$, we have an almost split sequence
\begin{equation*}
\xi:(0) \rightarrow W((s-1)p + a) \rightarrow W(sp + a) \oplus W((s-2)p + a) \rightarrow W((s-1)p +a) \rightarrow (0)
\end{equation*}
by \cite[4.1.2]{Farn1}, where $W(a) = (0)$. Since $F(\widehat{W}((s-1)p+a)) \cong W((s-1)p+a), F(\widehat{W}((s-2)p+a)) \cong W((s-2)p+a)$ and $F(\widehat{W}(sp+a)) \cong W(sp+a)$, \ref{thm: modgr has almost split sequences} and \ref{prop: indecomposables, projectives in modgr} show that there is an almost split sequence 
\begin{equation*}
\xi': (0) \rightarrow \widehat{W}((s-1)p + a)[\lambda] \rightarrow \widehat{W}(sp + a) \oplus \widehat{W}((s-2)p + a) [\nu] \rightarrow \widehat{W}((s-1)p +a) [\mu] \rightarrow (0)
\end{equation*}
in $\m G_1 T$ with $\lambda, \nu, \mu \in X(T)$, where $\widehat{W}(a):= 0$. Thus, $\tau_{G_1 T} (\widehat{W}((s-1)p +a)) [\mu]) \cong \widehat{W}((s-1)p + a)[\lambda]$. On the other hand, \ref{lem: tau for modules of complexity 1} yields $\alpha \in R$ such that $\tau_{G_1 T} \vert_{\Theta} = [p \alpha]$. Hence $\tau_{G_1 T} \vert_{\Theta} = [p \alpha ] = [\lambda - \mu]$ and \ref{prop: indecomposables, projectives in modgr}, (2) yields $p\alpha = \lambda - \mu$. In order to determine $\lambda, \mu$, we determine all submodules and factor modules of $\widehat{W}(sp+a)$ isomorphic to a shift of $\widehat{W}((s-1)p+a)$. Let $A$ be such a submodule. In the canonical basis $v_{a+1}, \ldots, v_{sp+a}$ of $\widehat{W}(sp+a)$, each $v_i$ has degree $(i, sp+a - i)$, so that the subspace $\widehat{W}(sp+a)_{(i, sp + a -i)}$ is one-dimensional with basis $v_i$. Hence any $G_1 T$-submodule of $\widehat{W}(sp+a)$ is the span of a subset of the $v_i$. As $\widehat{W}((s-1)p+a)$ is indecomposable and has dimension $(s-1)p$, \ref{lem: homogeneous submodules of V(sp + a) are spanned by vi,..., vi+k} provides a basis $v'_0, \ldots, v'_{(s-1)p - 1}$ of $A$ and $x, y \in \Z$ such that $\deg(v'_i) = (x + i, y + sp - 1 - i)$. Comparing the degrees of these basis elements with the degrees of the $v_i$ shows that there is $0 \leq i \leq p$ such that $A$ is spanned by $v_{a+1 + i}, \ldots, v_{(s-1)p + a + i}$. Suppose first that $0 \leq i \leq p - a - 2$. Then $e.v_{(s-1)p+a+i} = ((s-1)p+a+i+1)v_{(s-1)p+a+i+1}$ yields $v_{(s-1)p+a+i+1} \in A$, a contradiction. Thus, $i \geq p - a -  1$. If $p - a - 1 \leq i \leq p-1$, then $f.v_{a+ 1 + i} = (sp + i)v_{a +i}$ yields $v_{a+i} \in A$, a contradiction. As a result, we have $i = p$ and $A$ is spanned by $v_{p+a+1}, \ldots, v_{sp+a}$. Comparing the degrees of this basis of $A$ with the degrees of the canonical basis of $\widehat{W}((s-1)p+a)$, we see that $A \cong \widehat{W}((s-1)p+a)[(p,0)]$, so that $\lambda = (p, 0)$. \\ 
As $[\lambda - \mu ] = [p \alpha]$ and $\GL_n$ has roots $(1, -1), (-1,1)$, we have $\mu = (2p, -p)$ or $\mu = (0, p)$. Since $\widehat{W}((s-1)p+a)[\mu] \in \m M_1 D$, $\mu \neq (2p, -p)$; else the last base vector of the canonical basis of $\widehat{W}((s-1)p+a)[\mu]$ would have a negative coordinate. Thus, $\mu = (0, p)$.  
As a result, we have $\tau_{G_1 T} \vert_{\Theta} = [(p, -p)]$. By considering the degrees of the canonical basis of $\widehat{W}(sp+a)$ again, we see that not all weights of $\widehat{W}(sp+a)[(p, -p)], \widehat{W}(sp+a)[(-p, p)]$ are polynomial, so that  $\tau_{G_1 T}(\widehat{W}(sp+a)), \tau_{G_1 T}^{-1}(\widehat{W}(sp+a)) \notin \m M_1 D$. Now \ref{cor: mountains in components corresponding to homogeneous tubes} shows that $\Theta \cap \m M_1 D$ is equal to the wing of $\widehat{W}(sp+a)$. Now let $s = 1$. The arguments above in particular yield an almost split sequence
\begin{equation*}
(0) \rightarrow \widehat{W}((p + a)[(p, 0)] \rightarrow \widehat{W}(2p + a) \rightarrow \widehat{W}(p +a) [(0, p)] \rightarrow (0).
\end{equation*}
Applying the shift functor $[(-p, 0)]$ to this sequence, we see that $\tau_{G_1 T}(\widehat{W}(p +a) [(-p, p)]) = \widehat{W}(p +a)$, so that $\tau_{G_1 T} \vert \Theta = [(p, -p)]$ also holds in this case. Considering the degrees of the canonical basis for $\widehat{W}(p +a) [(-p, p)], \widehat{W}(p +a) [(p, -p)]$, we see that not all weights of these modules are polynomial, so that $\Theta \cap \m M_1 D$ is the wing of $\widehat{W}(p + a)$.\\
For $1 \leq s \leq 2$, the almost split sequences computed above now show that the leftmost quasi-simple module in the wing of $\widehat{W}(sp+a)$ is $\widehat{W}(p+a)[((s-1)p, 0)]$. For $s > 2$, one shows by induction that the $i$-th predecessor of $\widehat{W}(sp+a)$ on the sectional path starting in the quasi-socle of $\widehat{W}(sp+a)$ is $W((s-i)p+a)[(ip, 0)]$ by shifting the almost split sequence starting in $\widehat{W}((s-i)p+a)[(p, 0)]$ by $[((i-1)p, 0)]$. The definition of the $\Sl_2$-action on the basis vectors $v_{a+1}, \ldots, v_{p+a}$ of $\widehat{W}(p+a)$ shows that $v_{a+1}, \ldots, v_{p-1}$ span a simple submodule of $\widehat{W}(p+a)$ while each of the remaining basis vectors generates $\widehat{W}(p+a)$ as an $\Sl_2$-module, so that the socle of $\widehat{W}(p+a)$ is spanned by $v_{a+1}, \ldots, v_{p-1}$. As $\deg(v_{a+1}) = (p-1, a+1)$, we get $\F(\soc(\widehat{W}(p+a)[sp, -p])) = 0$, so that $\F(\widehat{W}(p+a)[sp, -p])= 0$ and $W(p+a)[((s-1)p, 0)]$ is projective in $\m M_1 D$ by \ref{lem: Ext-projective in T for graded algebras, almost split sequences in T}.  
\end{proof}

For an illustration of the following proposition, see Figure \ref{fig: component of type ZAdinf} below.

\begin{proposition}
Let $\Theta$ be the component of the stable Auslander-Reiten quiver of $\m G_1 T$ containing the module $\widehat{V}(sp + a)[(i, i)]$ for $0 \leq i \leq p - 1$. Then $\Theta$ is a $\ZAdinf$-component such that $\Theta = \Theta^o$. There is a column of simple module in $\Theta$ which is a symmetry axis with respect to $(-)^o$ and $\Theta \cap \m M_1 D$ consists of the modules on directed paths between $\widehat{V}(sp+a)[(i, i)]$ and $\widehat{V}(sp+a)^{o}[(i, i)]$. We also have $\tau_{G_1 T}(\widehat{V}(sp+a)[(i, i)]) = \widehat{V}((s+2)p+a)[(i-p, i-p)]$.

\label{prop: components of type ZAdinf for r = 1}
\end{proposition}

\begin{proof}
By \cite[Section 4.1]{Farn1}, $F(\Theta)$ is of type $\Z[\tilde{A}_{12}]$ and contains exactly one simple module $L(a)$. By \ref{prop: mod HT sum of blocks of mod H rtimes T}, $F$ induces a morphism of stable translation quivers $\Theta \rightarrow F(\Theta)$. Since $F(\Theta)$ has no quasi-simple modules and $\Theta$ is isomorphic to either $\ZAinf, \ZAdinf$ or $\Z[D_{\infty}]$ by \cite[3.4]{Farn4}, this shows that $\Theta$ is of type $\ZAdinf$. We show that there is an almost split sequence 
\begin{align*}
\zeta:(0) \rightarrow \widehat{V}(sp + a) & \rightarrow \widehat{V}((s-1)p + a)[( p , 0)] \oplus \widehat{V}((s-1)p + a)[( 0, p)]\\
 & \rightarrow  \widehat{V}((s-2)p+a)[(p , p )]\rightarrow (0)
\end{align*}
for $s > 1$
while the almost split sequence starting in $\widehat{V}(p+a)$ is
\begin{align*}
\zeta':(0) & \rightarrow \widehat{V}(p + a) \rightarrow \widehat{L}(a)[( p, 0)] \oplus \widehat{L}(a)[( 0, p)] \oplus \widehat{P}\\
 & \rightarrow  \widehat{V}(p+a)^o \rightarrow (0)
\end{align*}
with $\widehat{P}$ projective indecomposable. 
By \cite[V.5.5]{ARS1},  there is an almost split sequence 
\begin{equation*}
\xi:(0) \rightarrow \Rad(P) \rightarrow \Rad(P)/\soc(P) \oplus P \rightarrow P/\soc(P) \rightarrow (0)
\end{equation*}
in $\m U_0(\Sl_2)$, where $P$ is the projective cover of $L(p-a-2)$. By \cite[IV.3.8.3]{Erd1} in conjunction with \cite[2.4]{FarnRoehr1}, we have $\Rad(P)/\soc(P) \cong S \oplus S$ for some simple $U_0(\Sl_2)$-module $S$ such that $S \ncong L(p-a-2)$ and $S, L(p-a-2)$ are the only simple modules in the block of $S$. As $L(a)$ belongs to the block of $L(p-a-2)$, we get $S \cong L(a)$. Hence $\Rad(P)$ has dimension $p+a+1$ and $\soc(\Rad(P)) \cong L(p - a - 2)$, so that \ref{theorem: classification of U_0(sl2)-modules} yields $\Rad(P) \cong V(p+a)$. By \ref{thm: modgr has almost split sequences} in combination with \ref{prop: mod HT sum of blocks of mod H rtimes T}, the almost split sequence in $\m G_1 T$ starting in $\widehat{V}(p+a)$ is mapped to $\xi$ by $F$. Thus, the non-projective part of the middle term is a sum of two shifts of $\widehat{L}(a)$ and is obtained by factoring out the socle of $\widehat{V}(p+a)$. For the canonical basis $v_0, \ldots, v_{p+a}$ of $\widehat{V}(p+a)$, the action of $\Sl_2$ on the basis vectors shows that the socle is spanned by $v_{a+1}, \ldots, v_{p-1}$ and the direct summands of the factor module are spanned by the images of $v_0, \ldots, v_a$ and $v_p, \ldots, v_{p+a}$, respectively. Comparing the degrees of these vectors with the degrees in the canonical basis of $\widehat{L}(a)$, we see that $\widehat{V}(p+a)/ \soc(\widehat{V}(p+a)) \cong \widehat{L}(a)[( p, 0)] \oplus \widehat{L}(a)[( 0, p)]$. As the middle term of the almost split sequence starting in $\widehat{V}(p+a)$ is self-dual with respect to $(-)^o$, the right term of this sequence is $\widehat{V}(p+a)^o$, so that $\zeta'$ is almost split. By applying the shift functors $[(0, p)], [(p, 0)]$ to $\zeta'$, we see that $\widehat{V}(p+a)[(p, 0)], \widehat{V}(p+a)[(0, p)]$ are predecessors of $\widehat{V}(a)[(p, p)]$ in $\Gamma_s(G_1 T)$. As $\Theta$ is of type $\ZAdinf$, this shows that the middle term of the almost split sequence ending in $\widehat{V}(p+a)[(p, p)]$ has the claimed form. Inductively, to show that the almost split sequence for $s > 1$ has the claimed form, it suffices to show that if 
\begin{align*}
\xi': (0) \rightarrow X_s &\rightarrow \widehat{V}((s-1)p + a)[( p, 0)] \oplus \widehat{V}((s-1)p + a)[( 0, p)] \\
&\rightarrow  \widehat{V}((s-2)p+a)[(p, p)]\rightarrow (0)
\end{align*}  
is almost split in $\m G_1 T$, then $X_s \cong \widehat{V}(sp+a)$. As above, $F(\xi')$ is almost split in $\m U_0(\Sl_2)$ and by exactness, we have $\dim_k X_s = sp + a + 1$. Thus, \ref{theorem: classification of U_0(sl2)-modules} yields that $F(X_s) \cong V(sp+a) \cong F(\widehat{V}(sp+a))$ or $F(X_s) \cong V(sp+a)^{*} \cong F(\widehat{V}(sp+a)^o)$, so that $X_s$ is isomorphic to a shift of $\widehat{V}(sp+a)$ or $\widehat{V}(sp+a)^o$ by \ref{prop: indecomposables, projectives in modgr}. We show that $\widehat{V}(sp+a)$ does not have a submodule isomorphic to a shift of $\widehat{V}((s-1)p+a)^o$, so that $\widehat{V}(sp+a)^o$ does not have a factor module isomorphic to a shift of $\widehat{V}((s-1)p+a)$ and the first alternative has to apply. Suppose $A$ is a submodule of $\widehat{V}(sp+a)$ isomorphic to a shift of $\widehat{V}((s-1)p+a)^o$. Taking the canonical basis $v_0, \ldots, v_{sp+a}$ of $\widehat{V}(sp+a)$ and comparing the degrees of this basis with the degrees of the basis of $A$ provided by \ref{lem: homogeneous submodules of V(sp + a) are spanned by vi,..., vi+k}, we see that $A$ is spanned by $v_i, \ldots, v_{(s-1)p+a+i}$ for some $0 \leq i \leq p$. Using the action of $U_0(\Sl_2)$ on the $v_j$, we get the following statements:
\begin{enumerate}[(1)]
\item If $v_j \in A$ for some $0 \leq j \leq a$, then $v_0, \ldots, v_a \in A$,
\item if $v_j \in A$ for some $a+1 \leq j \leq p-1$, then $v_{a+1}, \ldots, v_{p-1}\in A$,
\item if $v_{j} \in A$ for some $(s-1)p+a+1 \leq j \leq sp - 1$, then $v_{(s-1)p+a+1}, \ldots, v_{sp -1} \in A$,
\item if $v_j \in A$ for some $sp \leq j \leq sp+a$, then $v_{sp}, \ldots, v_{sp+a} \in A$.
\end{enumerate}
Thus, $A$ is spanned by either $v_0, \ldots, v_{(s-1)p+a}$ or by $v_p, \ldots, v_{sp+a}$. However, we have $e.v_{(s-1)p+a} = ((s-1)p+a+1)v_{(s-1)p+a+1}$ and $f. v_p = ((s-1)p+a+1)v_{p-1}$, so that in both cases, $A$ is not a submodule, a contradiction. Hence $X_s \cong \widehat{V}(sp+a)[\lambda]$ for some $\lambda \in X(T)$. Since $\xi'$ is almost split, $\widehat{V}(sp+a)[\lambda]$ has factor modules $A_1 \cong \widehat{V}((s-1)p + a)[( p, 0)], A_2 \cong \widehat{V}((s-1)p + a)[( 0, p)]$. Comparing degrees of bases and applying the arguments on submodules from above to factor modules and kernels of the canonical projections, we see that $A_1$ resp. $A_2$ are spanned by the images of the basis vectors  $v_p, \ldots, v_{sp+a}$ resp. $v_0, \ldots, v_{(s-1)p+a}$ under the canonical projections and that $\lambda = 0$, so that $X_s \cong \widehat{V}(sp+a)$ and $\zeta$ is almost split. The statement about $\tau_{G_1 T}$ now follows by shifting the sequence for $s + 2$ by $[(i-p, i-p)]$. By applying the shift functor $[(lp, l'p)]$ to the almost split sequences $\zeta, \zeta'$, we see that for all $l, l' \in \Z$, the element directly above resp. below $\widehat{V}(sp+a)[(lp, l'p)]$ in the column of $\widehat{V}(sp+a)[(lp, l'p)]$ is $\widehat{V}(sp+a)[((l+1)p, (l'-1)p)]$ resp. $\widehat{V}(sp+a)[((l-1)p, (l'+1)p)]$. Thus, if $0 \leq j \leq s$ and we follow the upper resp. lower sectional path starting in $\widehat{V}(sp+a)$ for $j$ steps, we arrive at the module $\widehat{V}((s-j)p+a)[(jp, 0)]$ resp. $\widehat{V}((s-j)p+a)[(0, jp)]$. The modules above $\widehat{V}((s-j)p+a)[(jp, 0)]$ have the form $\widehat{V}((s-j)p+a)[((j+l)p, -lp)]$ while the modules below $\widehat{V}((s-j)p+a)[(0, jp)]$ have the form $\widehat{V}((s-j)p+a)[(-lp, (j+l)p)]$ for some $l \in \N$.  Thus, the degree of either $v_0$ or $v_{(s-j)p+a}$ in their canonical bases has a coordinate which is a negative multiple of $p$, so that these modules and their shifts by $[(i, i)]$ are not polynomial $G_1 T$-modules. All modules between $\widehat{V}((s-j)p+a)[(jp, 0)]$ and $\widehat{V}((s-j)p+a)[(0, jp)]$ in their column can be reached via directed paths starting in $\widehat{V}(sp+a)$. As all arrows on these paths correspond to irreducible epimorphisms, all modules obtained this way are polynomial $G_1 T$-modules by \ref{lem: subcategory of polynomial GrT-modules closed w. r. to operations}. The arguments above show in particular that there is a column of simple modules in $\Theta$ which is equal to $\{\widehat{L}(a)[(sp + jp, -jp)] \mid j \in \Z\}$ and that a module to the left of this column belongs to $\m M_1 D$ iff it lies on a directed path from $\widehat{V}(sp+a)$ to this column.  Since $S^o \cong S$ for every simple $G_1 T$-module $S$ by \cite[II.6.9(13)]{Jantz1}, we have $\Theta^o = \Theta$ and the column of simple modules in $\Theta$ is a symmetry axis with respect to $(-)^o$. As $V \in \m G_1 T$ is a polynomial $G_1 T$-module iff $V^o$ is a polynomial $G_1 T$-module, we get that any module to the right of the column of simple modules belongs to $\m M_1 D$ iff it lies on a directed path from the column of simple modules ending in $\widehat{V}(sp+a)^o$. As a result, $\Theta \cap \m M_1 D$ consists of all modules on directed paths from $\widehat{V}(sp+a)$ to $\widehat{V}(sp+a)^o$ for $i = 0$. For $1 \leq i \leq p-1$, the statement follows by applying the functor $[(i, i)]$ to the component containing $\widehat{V}(sp+a)$, noting that the above arguments about negative coordinates also apply in this case.
\end{proof}

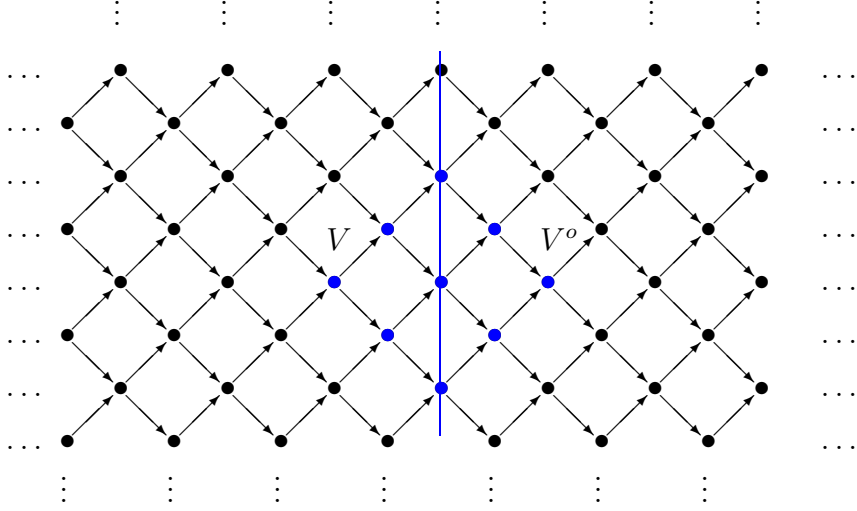
\begin{figure}[h]
\[ \begin{picture}(400, 180)
\multiput(40, 160)(40, 0)7{$\bullet$}

\multiput(20, 140)(40, 0)7{$\bullet$}

\multiput(40, 120)(40, 0)7{$\bullet$}

\multiput(20, 100)(40, 0)7{$\bullet$}

\multiput(40, 80)(40, 0)7{$\bullet$}

\multiput(20, 60)(40, 0)7{$\bullet$}

\multiput(40, 40)(40, 0)7{$\bullet$}

\multiput(20, 20)(40, 0)7{$\bullet$}


\multiput(45, 160)(40, 0)6 {\vector(1, -1){15}}

\multiput(25, 145)(40, 0)7 {\vector(1, 1){15}}

\multiput(25, 140)(40, 0)7 {\vector(1, -1){15}}

\multiput(45, 125)(40, 0)6 {\vector(1, 1){15}}

\multiput(45, 120)(40, 0)6 {\vector(1, -1){15}}

\multiput(25, 105)(40, 0)7 {\vector(1, 1){15}}

\multiput(25, 100)(40, 0)7 {\vector(1, -1){15}}

\multiput(45, 85)(40, 0)6{\vector(1,1){15}}

\multiput(45, 80)(40, 0)6 {\vector(1, -1){15}}

\multiput(25, 65)(40, 0)7{\vector(1, 1){15}}

\multiput(25, 60)(40, 0)7{\vector(1, -1){15}}

\multiput(45, 45)(40, 0)6{\vector(1,1){15}}

\multiput(45, 40)(40, 0)6{\vector(1,-1){15}}

\multiput(25, 25)(40, 0)7{\vector(1, 1){15}}


\multiput(40, 180)(40, 0)7{\vdots}

\multiput(0, 160)(0, -20)8{\ldots}
\multiput(305, 160)(0, -20)8{\ldots}

\multiput(20, 0)(40, 0)7{\vdots}

{\color{blue}
\put(160,120){$\bullet$}
\multiput(140, 100)(40, 0)2{$\bullet$} 
\multiput(120, 80)(40, 0)3{$\bullet$}
\multiput(140, 60)(40, 0)2{$\bullet$}
\put(160, 40){$\bullet$}

\put(162.5, 25){\line(0,1){145}}
}


\put(120, 95){$V$} 

\put(200, 95){$V^o$}

\end{picture}\]

\centering \caption{Component of $\Gamma_s(G_1 T)$ containing the module $V = \widehat{V}(2p+a)$. The blue vertices are $M_1 D$-modules and the blue line is the column of simple modules.}
\label{fig: component of type ZAdinf}
\end{figure}

By \ref{prop: mod HT sum of blocks of mod H rtimes T}, the shifts of $\widehat{V}(sp+a)$ with a $G_1 T$-structure are exactly those $\widehat{V}(sp+a)[\lambda]$ with $\lambda \in X(T)$ such that $\lambda \vert_{(\SL_2)_1 \cap T} = 0$. These are the elements of $pX(T) + \Z \det\vert_{T}$. The proof of \ref{prop: components of type ZAdinf for r = 1} shows that all these modules belong to components of the type described in \ref{prop: components of type ZAdinf for r = 1}. As modules of the form $V(sp+a)$ and their duals are the only $U_0(\mathfrak{sl}_2)$-modules contained in components of type  
$\Z[\tilde{A}_{12}]$ by \cite[Section 4]{Farn1}, we have found all of these modules for $r = 1$. \\
Since the restriction functor $\m M_1 D \rightarrow \m G_1 T$ is fully faithful, almost split sequences in $\m G_1 T$ such that all constituents lift to $M_1 D$ are also almost split in $\m M_1 D$. In the following, we determine the almost split sequences in $\m M_1 D$ which are not almost split in $\m G_1 T$. 
\begin{lemma}
There are almost split sequences 
\begin{align*}
\xi_1:(0) &\rightarrow  \widehat{W}((s-l)p + a))[(0, lp)] \\
&\rightarrow \widehat{V}((s-l)p + a)[(0,lp)]\oplus \widehat{W}((s - l - 1)p + a)[(0, (l+1)p)] \\ 
&\rightarrow \widehat{V}((s -l - 1)p +a)[(0, (l+1)p)] \rightarrow (0)
\end{align*}
for $0 \leq l  \leq s-1$ and
\begin{align*}
\xi_2:(0) &\rightarrow \widehat{V}((s-l)p - a - 2))[(a+1+lp, a+1)]^{o} \\
&\rightarrow \widehat{W}((s-l)p + a)[(lp,0)] \\
&\oplus \widehat{V}((s - l + 1)p - a -2 )[(a+1 + (l - 1)p , a+1 )]^{o} \\ 
&\rightarrow \widehat{W}((s -l + 1)p +a)[((l-1)p, 0)] \rightarrow (0)
\end{align*}
for $s > 1 $ and $1 \leq l \leq s- 1$ in $\m M_1 D$, where $\widehat{W}(a)=0$.

\label{lem: almost split sequences for n = 2, r = 1}
\end{lemma}

\begin{proof}
We have almost split sequences 
\begin{align*}
\xi_1':(0) &\rightarrow \widehat{V}((s-l + 1)p + a))[(-p, lp)] \\
&\rightarrow \widehat{V}((s- l)p + a)[(0,lp)]\oplus \widehat{V}((s - l)p + a )[(-p,(l+1)p)] \\ 
&\rightarrow \widehat{V}((s -l - 1)p +a)[(0,(l+1)p)] \rightarrow(0)
\end{align*}
and
\begin{align*}
\xi_2':(0) &\rightarrow \widehat{W}((s-l + 1)p + a))[(lp,-p)] \\
&\rightarrow \widehat{W}((s-l)p + a)[(lp,0)]\oplus \widehat{W}((s - l + 2)p + a)[ ((l-1)p, -p)]\\ 
&\rightarrow \widehat{W}((s -l + 1)p +a)[((l-1)p, 0)] \rightarrow (0)
\end{align*}
in $\m G_1 T$ by the preceding results and their proofs. We want to prove $\F(\xi_1') = \xi_1$ and $\F(\xi_2') = \xi_2$. We first show $A:=\F(\widehat{V}((s-l + 1)p + a))[(-p, lp)])= \langle v_{p+a+1}, \ldots, v_{(s-l+1)p+a} \rangle$, where $v_0, \ldots, v_{(s-l+1)p+a}$ is the canonical basis of $\widehat{V}((s-l + 1)p + a))[-(p, lp)]$. As every weight space of $\widehat{V}((s-l + 1)p + a))[(-p, lp)]$ is one-dimensional and $A$ is a homogeneous submodule, $A$ is spanned by a subset of the $v_i$. Since the first coordinates of their weights are negative, we have to delete at least $v_0, \ldots, v_{p-1}$, and since the $U_0(\Sl_2)$-module spanned by each of the vectors $v_p, \cdots v_{p+a}$ contains $v_{a+1}, \ldots, v_{p-1}$, we have to delete those as well, so that only the basis vectors spanning $\langle v_{p+a+1}, \ldots, v_{(s-l+1)p+a} \rangle$ remain. As $\langle v_{p+a+1}, \ldots, v_{(s-l+1)p+a} \rangle$ is stable under the $\Sl_2$-action on $\widehat{V}((s-l + 1)p + a))[(-p, lp)]$ and spanned by homogeneous elements, it is a homogeneous submodule. Since all weights of the basis vectors occurring are polynomial, we get $A = \langle v_{p+a+1}, \ldots, v_{(s-l+1)p+a} \rangle$. In particular, \ref{lem: Ext-projective in T for graded algebras, almost split sequences in T},(1) shows that the right-hand term of $\xi_1'$ is not $\Ext$-projective in $\m M_1 D$. Hence \ref{lem: Ext-projective in T for graded algebras, almost split sequences in T},(2) implies that $A$ is indecomposable, so that $F(A)$ is indecomposable by \ref{prop: indecomposables, projectives in modgr}. Since $v_{p+a+1}, \ldots, v_{2p-1}$ span a simple $\Sl_2$-submodule which is isomorphic to $L(p-a-2)$ by \ref{theorem: classification of U_0(sl2)-modules}, $L(p-a-2)$ occurs in the socle of $F(A)$. As $\dim_k A = (s-l)p$, \ref{theorem: classification of U_0(sl2)-modules} now yields $F(A) \cong W((s-l)p+a)$ or $F(A) \cong W((s-l)p+a)^{w_0}$. Using the action of $f$ on $A$, we see that the for $0 \leq i \leq s-l-1$, the vectors $v_{p+a+1 + ip}, \ldots, v_{2p+a+ip}$ span a $p$-dimensional $U_0(\langle f \rangle) \cong k[f]/(f^p)$-module which is generated by $v_{2p+a+ip}$ and is hence projective indecomposable for $U_0(\langle f \rangle)$. As $A \vert_{U_0(\langle f \rangle)}$ is the direct sum of these modules, $A \vert_{U_0(\langle f \rangle)}$ is projective, so that $f$ does not belong to the rank variety of $F(A)$. Since $f$ belongs to the rank variety of $W((s-l)p+a)^{w_0}$, we get $F(A) \cong W((s-l)p+a)$. Thus, $A$ is isomorphic to a shift of $\widehat{W}((s-l)p+a)$ and comparing the degrees in the canonical basis of $\widehat{W}((s-l)p+a)$ to those of $v_{p+a+1}, \ldots, v_{(s-l+1)p+a}$ yields $A \cong \widehat{W}((s-l)p+a)[(0, lp)]$. This also shows that $\F(\widehat{V}((s-l)p + a))[(-p, (l+1)p)]) \cong \widehat{W}((s-l-1)p+a)[(0, (l+1)p)]$. Now \ref{lem: Ext-projective in T for graded algebras, almost split sequences in T} shows that $\xi_1$ is almost split in $\m M_1 D$.
\\
We now compute $B:=\F(\widehat{W}(s-l+1)p+a)[(lp,-p)])$. As above, $B$ is spanned by a subset of the $v_i$ as a vector space. Since their weigths have a negative second coordinate, we have to delete at least the last $p$ base vectors $v_{(s-l)p+a+1}, \ldots, v_{(s-l+1)p+a}$. The $U_0(\Sl_2)$-module spanned by each of the vectors $v_{(s-l)p+1}, \ldots, v_{(s-l)p+a}$ contains the last $p$ basis vectors, so they also have to be deleted. As the remaining vectors $v_{a+1}, \ldots, v_{(s-l)p}$ span a homogeneous submodule of $W(s-l+1)p+a)[(lp,-p)]$, we get $B = \langle v_{a+1}, \ldots, v_{(s-l)p}\rangle$. As above, $B$ is indecomposable. As $\dim_k(B) = (s-l+1)p - a -1$, \ref{theorem: classification of U_0(sl2)-modules} shows that $F(B) \cong V((s-l+1)p-a-2)$ or $F(B)\cong V((s-l+1)p-a-2)^o$. As $\langle v_{a+1}, \ldots, v_{p-1} \rangle$ is a simple $\Sl_2$-submodule isomorphic to $L(p-a-2)$, we get $F(B)\cong V((s-l+1)p-a-2)^o$ since $L(p-a-2)$ does not occur in the socle of  $V((s-l+1)p-a-2)$. Now \ref{prop: indecomposables, projectives in modgr} yields $B \cong \widehat{V}((s-l+1)p-a-2)^o[\lambda]$ for some $\lambda \in X(T)$. Comparing the degrees of $v_{a+1}, \ldots v_{(s-l)p}$ with those of the canonical basis of $\widehat{V}((s-l+1)p-a-2)^o$ yields $\lambda = (a+1+lp, a+1)$. Thus, $B \cong \widehat{V}((s-l+1)p-a-2)^o[(a+1+lp, a+1)]$. This also shows  $\F(\widehat{W}((s - l + 2)p + a)[ ((l-1)p, -p)]) \cong \widehat{V}((s - l)p - a -2 )[(a+1 + (l - 1)p , a+1 )]^{o}$. Now \ref{lem: Ext-projective in T for graded algebras, almost split sequences in T} shows that $\xi_2$ is almost split in $\m M_1 D$.    
\end{proof}

\begin{lemma}
There are almost split sequences 
\begin{align*}
\xi_1:(0) &\rightarrow \widehat{W}((s-l)p + a))^{w_0}[(lp, 0)] \\
&\rightarrow \widehat{V}((s-l)p + a)[(lp, 0)]\oplus \widehat{W}((s - l - 1)p + a)^{w_0}[((l+1)p, 0)] \\ 
&\rightarrow \widehat{V}((s -l - 1)p +a)[((l+1)p, 0)] \rightarrow (0)
\end{align*}
for $0 \leq l \leq s-1$ and
\begin{align*}
\xi_2:(0) &\rightarrow \widehat{V}((s-l)p - a - 2))[(a+1, a+1 + lp)]^{o} \\
&\rightarrow \widehat{W}((s-l)p + a)^{w_0}[(0,lp)]\\
& \oplus \widehat{V}((s - l + 1)p - a - 2 )[(a+1 , a+1 + (l-1)p )]^{o} \\ 
&\rightarrow \widehat{W}((s -l + 1)p +a)^{w_0}[(0 , (l-1)p)] \rightarrow (0)
\end{align*}
for $s > 1$ and $1 \leq l \leq s- 1$ in $\m M_1 D$, where $\widehat{W}(a)=0$.
\label{lem: almost split sequences for n = 2, r = 1: twist with w_0}
\end{lemma}

\begin{proof}
Twisting with $w_0$ is an auto-equivalence of $\m M_1 D$ mapping the almost split sequences from \ref{lem: almost split sequences for n = 2, r = 1} to $\xi_1, \xi_2$. The result now follows.
\end{proof}

\begin{lemma}
There is an almost split sequence
\begin{align*}
(0) &\rightarrow \widehat{V}(sp - a - 2)[(a+1, a+1)]^{o} \\ 
&\rightarrow \widehat{W}(sp + a)\oplus \widehat{W}(sp + a)^{w_o} \\
&\rightarrow \widehat{V}(sp +a) \rightarrow (0)
\end{align*}
in $\m M_1 D$ for $s \geq 1$.
\label{lem: almost split sequences for n = 2, r = 1: sequence ending in V(sp+a)}
\end{lemma}

\begin{proof}
By the proof of \ref{prop: components of type ZAdinf for r = 1}, there is an almost split sequence 
\begin{align*}
(0) &\rightarrow \widehat{V}((s+2)p + a)[-(p, p)] \\
&\rightarrow \widehat{V}((s+1)p + a)[(-p,0)]\oplus \widehat{V}((s+1)p + a)[(0,-p)]\\ 
&\rightarrow \widehat{V}(sp +a)\rightarrow (0)
\end{align*}
in $\m G_r T$. In the proof of \ref{lem: almost split sequences for n = 2, r = 1}, it was shown that $\F(\widehat{V}((s+1)p + a)[(-p,0)]) \cong \widehat{W}(sp+a)$. Since $\widehat{V}((s+1)p+a)[(-p, 0)]^{w_0} \cong \widehat{V}((s+1)p+a)[(0, -p)]$, this also shows $\F(\widehat{V}((s+1)p + a)[(0,-p)]) \cong \widehat{W}(sp+a)^{w_0}$. We compute $C = \F(\widehat{V}((s+2)p + a)[-(p, p)])$. Taking the canonical basis $v_0, \ldots, v_{(s+2)p+a}$ of $\widehat{V}((s+2)p + a)[-(p, p)]$, the $\Sl_2$-action on the base vectors shows that $C$ is spanned by $v_{p+a+1}, \ldots, v_{(s+1)p-1}$ as all other base vectors generate submodules which are not polynomial. Since $\dim_k (C) = sp - a -1$ and $v_{p+a+1}, \ldots, v_{2p-1}$ span an $\Sl_2$-module isomorphic to $L(p-a-2)$, \ref{theorem: classification of U_0(sl2)-modules} yields $F(C) \cong V(sp-a-2)^* \cong F(\widehat{V}(sp-a-2)^o)$. An application of \ref{prop: indecomposables, projectives in modgr} and a comparison of degrees of the basis $v_{p+a+1}, \ldots, v_{(s+1)p-1}$ of $C$ and the canonical basis of $\widehat{V}(sp-a-2)^o$ yields $C \cong \widehat{V}(sp - a - 2)[(a+1, a+1)]^{o}$. The claim now follows from \ref{lem: Ext-projective in T for graded algebras, almost split sequences in T}.   
\end{proof}

By using the results about stable AR-components of $\m G_1 T$ and almost split sequences in $\m M_1 D$, we are able to determine the Auslander-Reiten quiver for the blocks of $S_{sp+a}(G_1 T)$ containing $\widehat{V}(sp+a)$.
\begin{theorem}
The component of the AR-quiver of $S_{sp+a}(G_1 T)$ for $s \geq 1$ and $0 \leq a \leq p-2$ containing $\widehat{V}(sp+a)$ has the following form:

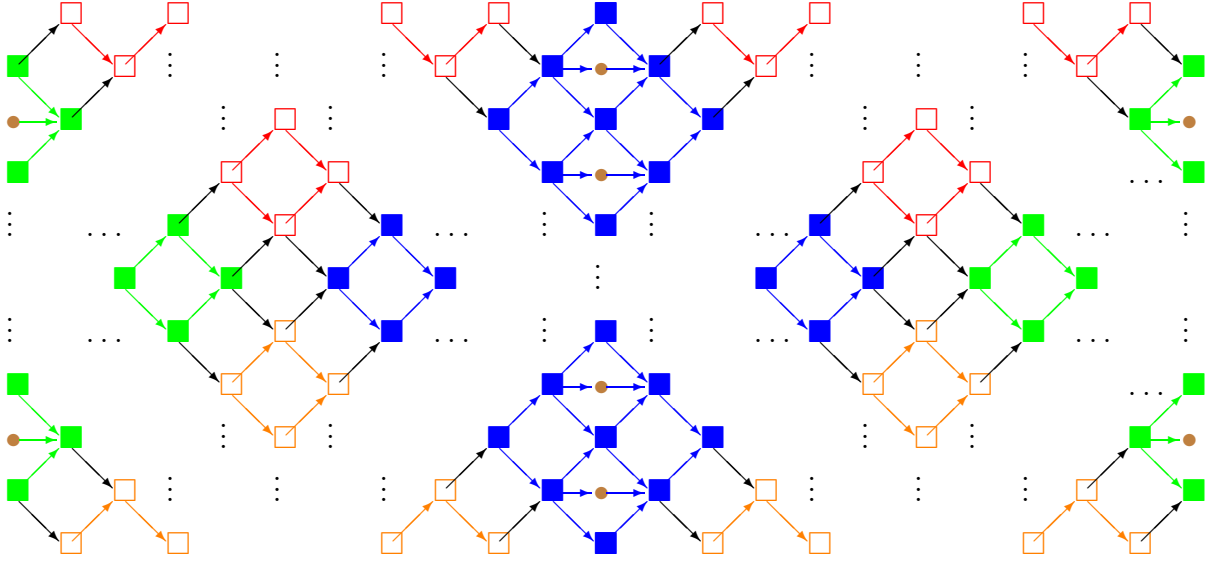
\begin{figure}[H]
\[ \begin{picture}(500, 200)



{\color{green}
\put(80, 100){$\blacksquare$}

\put(40, 100){$\blacksquare$}

\multiput(0, 20)(0, 40)2{$\blacksquare$}

\multiput(0, 140)(0, 40)2{$\blacksquare$}

\multiput(20, 40)(0, 120)2{$\blacksquare$}

\multiput(60, 80)(0, 40)2{$\blacksquare$}

{\color{brown}
\multiput(0, 40)(0, 120)2{$\bullet$}
}

\multiput(5, 60)(0, 120)2{\vector(1, -1){15}}

\multiput(5, 25)(0, 120)2{\vector(1, 1){15}}

\put(65, 85){\vector(1, 1){15}}

\put(65, 120){\vector(1, -1){15}}

\multiput(5, 43)(0, 120)2{\vector(1,0){14}}

\put(45, 105){\vector(1,1){15}}
\put(45, 100){\vector(1,-1){15}}}

\multiput(30,80)(0, 40)2{$\ldots$}

\multiput(0, 80)(0, 40)2{$\vdots$}


{\color{orange}
\multiput(80, 60)(40, 0)2{$\square$}

\multiput(20,0)(40, 0)2{$\square$}

\multiput(140,0)(40, 0)2{$\square$}

\multiput(40, 20)(120, 0)2{$\square$}

\put(100, 80){$\square$}

\put(100,40){$\square$}


\put(105, 80){\vector(1, -1){15}}
\put(85, 65){\vector(1, 1){15}}

\multiput(25, 5)(120, 0)2{\vector(1, 1){15}}
\multiput(45, 20)(120, 0)2{\vector(1, -1){15}}

\put(85, 60){\vector(1,-1){15}}

\put(105, 45){\vector(1,1){15}}}


\multiput(80, 40)(40, 0)2{\vdots}

\multiput(60, 20)(40, 0)3{\vdots}


{\color{red}
\multiput(80, 140)(40, 0)2{$\square$}

\multiput(20, 200)(40, 0)2{$\square$}

\multiput(140, 200)(40, 0)2{$\square$}

\multiput(40, 180)(120, 0)2{$\square$}

\put(100, 120){$\square$}

\put(100, 160){$\square$}

\put(85, 140){\vector(1, -1){15}}
\put(105, 125){\vector(1,1){15}}

\multiput(45, 185)(120, 0)2{\vector(1, 1){15}}
\multiput(25, 200)(120, 0)2{\vector(1, -1){15}}

\put(85, 145){\vector(1,1){15}}

\put(105, 160){\vector(1, -1){15}}}

\multiput(80, 160)(40, 0)2{\vdots}

\multiput(60, 180)(40, 0)3{\vdots}


{\color{blue}

\put(120, 100){$\blacksquare$}

\put(160, 100){$\blacksquare$}

\multiput(220 ,0)(0, 40)3{$\blacksquare$}
\multiput(220 , 120)(0 , 40)3{$\blacksquare$}

\multiput(180, 40)(0, 120)2{$\blacksquare$}
\multiput(200, 20)(0, 160)2{$\blacksquare$}
\multiput(200, 60)(0, 80)2{$\blacksquare$}

\multiput(140, 80)(0, 40)2{$\blacksquare$}

{\color{brown}
\multiput(220, 20)(0, 160)2{$\bullet$}
\multiput(220, 60)(0, 80)2{$\bullet$}
}

\put(125, 100){\vector(1, -1){15}}
\put(125, 105){\vector(1, 1){15}}

\multiput(205, 25)(0, 160)2{\vector(1,1){15}}
\multiput(205, 65)(0, 80)2{\vector(1,1){15}}
\multiput(205, 20)(0, 160)2{\vector(1,-1){15}}
\multiput(205, 60)(0, 80)2{\vector(1,-1){15}}
\multiput(185, 40)(0, 120)2{\vector(1, -1){15}}
\multiput(185, 45)(0, 120)2{\vector(1,1){15}}
\multiput(205, 23)(0, 160)2{\vector(1,0){14}}
\multiput(205, 63)(0, 80)2{\vector(1,0){14}}
\put(145, 85){\vector(1,1){15}}
\put(145, 120){\vector(1,-1){15}}
} 


\multiput(160,80)(0, 40)2{$\ldots$}

\multiput(200, 80)(0, 40)2{$\vdots$}

\put(220, 100){$\vdots$}


\multiput(5, 185)(20,-20)2{\vector(1,1){15}}
\multiput(65, 125)(20,-20)4{\vector(1,1){15}}
\multiput(185, 5)(-20, 20)2{\vector(1,1){15}}

\multiput(5, 20)(20,20)2{\vector(1, -1){15}}
\multiput(65, 80)(20,20)4{\vector(1,-1){15}}
\multiput(185, 200)(-20,-20)2{\vector(1,-1){15}}



{\color{blue}
\put(320, 100){$\blacksquare$}

\put(280, 100){$\blacksquare$}

\multiput(220, 0)(0, 40)2{$\blacksquare$}

\multiput(240, 60)(0, 80)2{$\blacksquare$}

\multiput(260, 40)(0, 120)2{$\blacksquare$}
\multiput(220, 160)(0, 40)2{$\blacksquare$}

\multiput(240, 20)(0, 160)2{$\blacksquare$}

\multiput(300, 80)(0, 40)2{$\blacksquare$}

\multiput(225, 40)(0, 160)2{\vector(1, -1){15}}

\multiput(225, 80)(0, 80)2{\vector(1,-1){15}}
\multiput(245, 60)(0, 120)2{\vector(1, -1){15}}

\multiput(225, 5)(0, 160)2{\vector(1, 1){15}}

\multiput(245, 25)(0, 120)2{\vector(1, 1){15}}

\multiput(225,45)(0, 80)2{\vector(1,1){15}}

\put(305, 85){\vector(1, 1){15}}

\put(305, 120){\vector(1, -1){15}}

\multiput(225, 23)(0, 160)2{\vector(1,0){14}}

\multiput(225, 63)(0, 80)2{\vector(1,0){14}}
\put(285, 105){\vector(1,1){15}}
\put(285, 100){\vector(1,-1){15}}
}

\multiput(280,80)(0, 40)2{$\ldots$}

\multiput(240, 80)(0, 40)2{$\vdots$}


{\color{orange}
\multiput(320, 60)(40, 0)2{$\square$}

\multiput(260,0)(40, 0)2{$\square$}

\multiput(380,0)(40, 0)2{$\square$}

\multiput(280, 20)(120, 0)2{$\square$}

\put(340, 80){$\square$}

\put(340,40){$\square$}


\put(345, 80){\vector(1, -1){15}}
\put(325, 65){\vector(1, 1){15}}

\multiput(265, 5)(120, 0)2{\vector(1, 1){15}}
\multiput(285, 20)(120, 0)2{\vector(1, -1){15}}

\put(325, 60){\vector(1,-1){15}}

\put(345, 45){\vector(1,1){15}}}


\multiput(320, 40)(40, 0)2{\vdots}

\multiput(300, 20)(40, 0)3{\vdots}


{\color{red}
\multiput(320, 140)(40, 0)2{$\square$}

\multiput(260, 200)(40, 0)2{$\square$}

\multiput(380, 200)(40, 0)2{$\square$}

\multiput(280, 180)(120, 0)2{$\square$}

\put(340, 120){$\square$}

\put(340, 160){$\square$}

\put(325, 140){\vector(1, -1){15}}
\put(345, 125){\vector(1,1){15}}

\multiput(285, 185)(120, 0)2{\vector(1, 1){15}}
\multiput(265, 200)(120, 0)2{\vector(1, -1){15}}

\put(325, 145){\vector(1,1){15}}

\put(345, 160){\vector(1, -1){15}}
}

\multiput(320, 160)(40, 0)2{\vdots}

\multiput(300, 180)(40, 0)3{\vdots}


{\color{green}

\put(360, 100){$\blacksquare$}

\put(400, 100){$\blacksquare$}

\multiput(440 ,20)(0, 40)2{$\blacksquare$}
\multiput(440 , 140)(0 , 40)2{$\blacksquare$}

\multiput(420, 40)(0, 120)2{$\blacksquare$}

\multiput(380, 80)(0, 40)2{$\blacksquare$}

{\color{brown}
\multiput(440, 40)(0, 120)2{$\bullet$}
}

\put(365, 100){\vector(1, -1){15}}
\put(365, 105){\vector(1, 1){15}}

\multiput(425, 45)(0, 120)2{\vector(1,1){15}}
\multiput(425, 40)(0, 120)2{\vector(1,-1){15}}

\multiput(425, 43)(0, 120)2{\vector(1,0){14}}

\put(385, 85){\vector(1,1){15}}
\put(385, 120){\vector(1,-1){15}}
} 


\multiput(400,80)(0, 40)2{$\ldots$}

\multiput(440, 80)(0, 40)2{$\vdots$}

\multiput(420, 60)(0, 80)2{$\ldots$}


\multiput(245, 185)(20, -20)2{\vector(1,1){15}}
\multiput(305, 125)(20,-20)4{\vector(1,1){15}}
\multiput(425, 5)(-20,20)2{\vector(1,1){15}}

\multiput(245, 20)(20, 20)2{\vector(1, -1){15}}
\multiput(305, 80)(20,20)4{\vector(1,-1){15}}
\multiput(425, 200)(-20,-20)2{\vector(1,-1){15}}

\end{picture}\]

\end{figure}

Here, the leftmost and the rightmost column are identified, the brown dots are projective-injective, the blue squares are the $M_1 D$-part of the $G_1 T$-component containing $\widehat{V}(sp+a)$, the green squares are the $M_1 D$-part of the $G_1 T$-component containing $\widehat{V}(sp - a -2)[(a+1,a+1)]$, the red squares on the left side are the $M_1 D$-part of the $G_1 T$-component containing $\widehat{W}(sp + a)$, the orange squares on the left side are the $M_1 D$-part of the $G_1 T$-component containing $\widehat{W}(sp + a)^{w_0}$ and the red and orange squares on the right side are the duals of those on the left side. All colored arrows represent morphisms which are irreducible in $\m G_1 T$, while the black arrows represent morphisms which are irreducible in $\m M_1 D$, but not in $\m G_1 T$. Thus, upon deleting projective-injective vertices,the underlying directed graph of the component is isomorphic to $\Z[A_{2s+1}]/ \langle \tau^{2s+1}\rangle$.
\end{theorem}

\begin{proof}
For the left part of the quiver, one uses the shape of the $M_r D$-part of the $G_1 T$-components containing $\widehat{W}(sp+a), \widehat{V}(sp+a), \widehat{W}(sp+a)^{w_0}$ and $\widehat{V}(sp-a-2)^{o}[(a+1, a+1)]$ and the almost split sequences from \ref{lem: almost split sequences for n = 2, r = 1},  \ref{lem: almost split sequences for n = 2, r = 1: sequence ending in V(sp+a)} and \ref{lem: almost split sequences for n = 2, r = 1: twist with w_0}. The right part is now obtained from the left part by contravariant duality.
\end{proof}

\begin{remark}
Using calculations similar to those in the proofs of this section, one can show that $\widehat{W}(sp+a)^o \cong \widehat{W}((s+1)p-a-2)^{w_0}[(a+1, a+ 1 - p)]$ and $(\widehat{W}(sp+a)^{w_0})^o \cong \widehat{W}((s+1)p-a-2)[(a+1 - p, a+ 1)]$ and use this to give a different description of the modules in the right part of the component given above.
\end{remark}
\\

\ref{prop: equivalence between blocks of M_r-s D and M_r D} shows that components of this shape also occur for $r > 1$. We will show that all components of $S_d(G_1 T)$ are shifts of a component of this shape.

For this, we first show that some shifts induce Morita-equivalences between blocks of different $S_d(G_1 T)$.  

\begin{proposition}
Let $s \geq 0$ and $0 \leq a \leq p -2$. Then for $1 \leq i \leq p - a - 2$, the shift functor $[(i,i)]$ restricts to an equivalence of categories between the categories of finite dimensional modules belonging to the block of $S_{sp+a}(G_1 T)$ containing $\widehat{V}(sp + a)$ and a block of $S_{sp+a+2i}(G_1 T)$.
\label{prop: [-(i,i)] is a Morita-equivalence from block containing V(sp+a)}
\end{proposition}

\begin{proof}
Since the block contains the simple module of highest weight $(sp+a, 0)$, $\rho = (\frac{1}{2}, -\frac{1}{2})$, $W = \{1, w_0\}$ and $\det \vert_{T} = (1, 1)$, this follows from an application of \cite[4.2]{Dre1} to our situation.
\end{proof}

By counting the number of shifts and comparing this to the number of non-semisimple blocks of $S_{sp+a}(G_1 T)_1$ determined in \cite{DNP1}, we see that every non-semisimple block arises this way.

\begin{corollary}
Let $b$ be a non-simple block of $S_d(G_1 T)$. Then $b$ is the shift of the block of $S_d(G_1 T)$ containing $\widehat{V}(d')$ for some $d' \leq d$ such that $d' = sp + a$, $0 \leq a \leq p-2$. 
\end{corollary}

\begin{proof}
Write $d = sp + a$ with $0 \leq a \leq p - 1$ and $s > 1$. We first consider the case $a \neq p-1$. By \ref{prop: [-(i,i)] is a Morita-equivalence from block containing V(sp+a)}, the shifts by $[(i, i)]$ define a Morita-equivalence of the block of $S_{sp+a-2i}(G_1 T)$ containing the module $\widehat{V}(sp + a - 2i)$  to a block of $S_d(G_1 T)$ for $0 \leq i \leq \lfloor \frac{a}{2} \rfloor$. By the same token, the shifts by $[(a + i + 1, a + i + 1)]$ define a Morita-equivalence of the block of $S_{(s-1)p+ (p - a - 2) - 2i}(G_1 T)$ containing $\widehat{V}((s-1)p + (p-a-2)-2i)$ to a block of $S_d(G_1 T)$ for $1 \leq i \leq \lfloor \frac{p-a-2}{2} \rfloor$. These blocks are pairwise distinct since the indecomposable modules of largest dimension belonging to them are $\widehat{V}(sp + a - 2i)$ with dimension $sp+a-2i+1$ resp. $\widehat{V}((s-1)p + (p-a-2)-2i)$ with dimension $(s-1)p+ (p-a-2)-2i+1$. For $a = p -1$, only the first kind of blocks occurs with $1 \leq i \leq \frac{p-1}{2}$. Thus, we have found $\frac{p-1}{2}$ distinct blocks in both cases. By \cite[2.2]{DNP1}, we have found all blocks not associated to a shift of a Steinberg module, hence all non-semisimple blocks. If now $d = a$ with $0 \leq a \leq p-1$, then $S_d(G_1 T)$ is semisimple by \cite[Theorem 3]{DN1}. If $d = p + a$ for $0 \leq a \leq p-1$, then we get all blocks not associated to Steinberg modules as above, but only the first kind of blocks is non-semisimple since the second kind of blocks is a shift of a block of $S_{d'}(G_1 T)$ with $d' \leq p - 1$.
\end{proof}

Now let $r, n$ be arbitrary again. By reducing to the case $r = 1$, we are able to determine the $M_r D$-parts of $G_r T$-components $\Theta$ whose restrictions to $G_r$ have type $\Z[\tilde{A}_{12}]$. Since the restriction to $(\SL_n)_r$ induces a morphism of stable translation quivers by \ref{prop: mod HT sum of blocks of mod H rtimes T}, so that the restriction of $\Theta$ is a component of $\Gamma_s((\SL_n)_r)$, \cite[4.1]{Farn3} implies that the restriction of $\Theta$ to $(\SL_n)_r$ is also of type $\Z[\tilde{A}_{12}]$. Now \cite[4.2]{Farn3} implies $n = 2$ as $\SL_n$ is almost simple.

\begin{corollary}
Suppose that $F(\Theta)$ is of type $\Z[\tilde{A}_{12}]$ and $\Theta \cap \m M_r D \neq \emptyset$. Then $\Theta$ is of type $\ZAdinf$, $\Theta^o = \Theta$ and there is a column of simple modules in $\Theta$ which is a symmetry axis with respect to $(-)^o$.  There is an $M_r D$-module $V \in \Theta$ such that $\m M_r D \cap \Theta$ consists of all modules on directed paths from $V$ to $V^{o}$.  
\end{corollary}

\begin{proof}
Note that taking the Morita-equivalence in \ref{prop: equivalence between blocks of M_r-s D and M_r D} followed by the forgetful functor $\m G_r T \rightarrow \m (\SL_2)_r$ is the same as taking the forgetful functor followed by the Morita-equivalence of \cite[5.1]{Farn2}. As the Morita-equivalence of \ref{prop: equivalence between blocks of M_r-s D and M_r D} maps the $M_{r-s}D$-modules in a $G_{r-s}T$-block to the $M_r D$-modules in a $G_r T$-block and commutes with the forgetful functors, we can use the arguments of \cite[5.4]{Farn2} to reduce to the case $r=1$. The result now follows from \ref{prop: components of type ZAdinf for r = 1} and its succeeding remarks. 

\end{proof}

\section{Polynomial representations of $B_r T$}
In contrast to the ordinary Schur algebra, the algebra $S_d(G_r T)$ is not in general quasi-hereditary, see \cite[Section 7]{DNP2}. However, as we show below, the algebras $S_d(B_r T) $ are directed, i.e. quasi-hereditary with simple standard modules. See \cite{DlRi1} for the definition of a quasi-hereditary algebra.
\begin{proposition}
Let $G \subseteq \GL_n$ be a reductive group, $T \subseteq G$ a maximal torus and $B$ a Borel subgroup of $G$ containing $T$. Then the algebra $S_d(B_r T)$ is quasi-hereditary and directed.
\end{proposition}

\begin{proof}
Let $\lambda$ be a polynomial weight for $T$ of degree $d$. By \cite{Jantz1}, II.9.5 and the remark before II.9.3, all weights $\mu$ of the projective cover $\widehat{Z}_r(\lambda)$ of $k_\lambda$ in $\m B_r T$ satisfy $\mu \leq \lambda$ and the weight space for the weight $\lambda$ is 1-dimensional. Since every module with top $k_{\lambda}$ is a factor module of $\widehat{Z}_r(\lambda)$, the ordering is adapted (see the remarks below \cite[Lemma 1.2]{DlRi1} for the definition of adapted ordering). Since $\mathcal{G}_{\overline{B_r T}}(\widehat{Z}_r(\lambda))$ is projective in $\m \overline{B_r T}$ and has simple top $k_{\lambda}$, it is the projective cover of $k_\lambda$ in $\m \overline{B_r T}$ and we get $\mathcal{G}_{\overline{B_r T}}(\widehat{Z}_r(\lambda)) \cong \Delta(\lambda)$ by definition, see \cite[Section 1]{DlRi1}. As the $\lambda$-weight space of $\mathcal{G}_{\overline{B_r T}}(\widehat{Z}_r(\lambda))$ is one-dimensional, $k_{\lambda}$ occurs only once as a composition factor of $\mathcal{G}_{\overline{B_r T}}(\widehat{Z}_r(\lambda))$, so that $\mathcal{G}_{\overline{B_r T}}(\widehat{Z}_r(\lambda))$ has endomorphism ring isomorphic to $k$ by \cite[1.3]{DlRi1}. Now \cite[Theorem 1]{DlRi1} implies that $S_d(B_r T)$ is quasi-hereditary. Relative to the reverse ordering, $S_d(B_r T)$ obtains the structure of a quasi-hereditary algebra with simple standard modules by setting $\Delta(\lambda) = k_{\lambda}$, so that $S_d(B_r T)$ is directed. 
\end{proof}
Now let $T \subseteq \GL_2$ be the torus of diagonal matrices, $U \subseteq \GL_2$ be the group of upper triangular unipotent matrices, $B = U \rtimes T \subseteq \GL_2$ be the subgroup of upper triangular matrices and $L_r D = \overline{B_r T}$.
It is well known that $U \cong G_a$, $k[U_r] \cong k[X]/(X^{p^r})$ and $kU_r \cong k[X_1, \ldots, X_r]/(X_1^p, \ldots, X_r^p)$, see for example \cite[Section 2]{Farn2}.  By considering the action of $T$ on $U_r$ and the induced actions on $k[U_r]$ and $kU_r$, we see that the $\Z^2 \cong X(T)$-grading on $k[U_r]$ is given by $\deg(x_i) =p^{i-1}(1, -1)$, where $x_i$ is the residue class of $X_i$ in $k[U_r]$. Then $\m B_r T = \m U_r \rtimes T$ is the category of $\Z^2$-graded $kU_r$-modules. In this section, we consider all algebras $k[X_i, \ldots, X_j]/(X_i^p, \ldots X_j^p), 1 \leq i \leq j \leq r$, to be $\Z^2$-graded with these degrees.  
In this case, we can extend our results about modules of complexity one to tensor products of these modules with other indecomposable modules.

If $M, N$ as in the next result are indecomposable, general theory (\cite[I.10E]{CR1}) in combination with \ref{prop: indecomposables, projectives in modgr} shows that $M \otimes_k N$ is an indecomposable $B_r T$-module as $k$ is algebraically closed. 

\begin{proposition}
Let $M \in \m B_{r-h}T$ be indecomposable such that $cx_{B_{r-h}T}(M) = 1$ and let $N$ be a finite-dimensional non-projective indecomposable graded module over the algebra $k[X_{r-h+1}, \ldots, X_r]/(X_{r-h+1}^{p}, \ldots, X_r^{p})$.
Suppose the component $\Theta \subseteq \Gamma_s(B_r T)$ containing $M \otimes_k N$ is regular. Then $\Theta \cong \ZAinf$ and the quasi-length of elements of $\Theta \cap \m L_r D$ is bounded. If $M \otimes_k N$ is quasi-simple, $\Theta \cap \m L_r D$ is finite. 
\label{prop: tensor product of cx 1 => quasilength of polynomial modules in component bounded; quasi-simple => only finitely many polynomial modules}
\end{proposition}
\begin{proof}
As $\Theta$ is regular, we have $cx_{k[X_{r-h+1}, \ldots, X_r]/(X_{r-h+1}^{p}, \ldots, X_r^{p})}(N) \geq 1$ and thus $cx_{B_r T}(M \otimes_k N) = cx_{B_{r-h}T}(M) + cx_{k[X_{r-h+1}, \ldots, X_r]/(X_{r-h+1}^{p}, \ldots, X_r^{p})}(N) \geq 2$ (see for example \cite[3.2.15]{Kuel1}), so that $\Theta$ is not periodic. Now \cite[Theorem 1]{Erd2} implies $ \Theta \cong \ZAinf$ as $k[X_1, \ldots, X_r]/(X_1^p, \ldots, X_r^p)$ has wild representation type for $r>1$ since $p\geq 3$.
 
We show that the $\tau_{B_r T}$-orbit of $M \otimes_k N$ contains only finitely many $L_r D$-modules. As $\m L_r D$ is closed with respect to submodules and factor modules, the result then follows from \ref{lem: only finitely many quasi-simple modules in theta cap mod MrD => theta cap mod MrD finite}.

Let $P = (P_l)_{l \in \N_0}$ be a minimal projective resolution of $M$ as a graded $k[X_1, \ldots ,X_{r-h}]/(X_1^p, \ldots, X_{r-h}^p)$-module and $Q = (Q_l)_{l \in \N_0}$ be a minimal projective resolution of $N$ as a graded $k[X_{r-h+1}, \ldots, X_r]/(X_{r-h+1}^{p}, \ldots, X_r^{p})$-module. Then the complex $P \otimes_k Q$ is a minimal projective resolution of $M \otimes_k N$ as a $B_r T$-module, where
\begin{equation*}
(P \otimes_k Q)_l = \bigoplus_{i+j=l} P_i \otimes_k P_j
\end{equation*}
and
\begin{equation*}
d_{{P \otimes_k Q},l}= \bigoplus_{i+j = l} d_{P, i} \otimes \id_{Q_j} + (-1)^i \id_{P_i} \otimes d_{Q, j}
\end{equation*} 
for all $l > 0$. Thus, for all $i>0$, we get $\Omega^{i}_{B_{r-h}T}(M) \otimes_k \Omega^1_{k[X_{r-h+1}, \ldots, X_r]/(X_{r-h+1}^{p}, \ldots, X_r^{p})}(N) \subseteq \Omega^{i}_{B_r T}(M \otimes_k N)$. Let $\lambda \in \Z^2$ such that $\Omega^1_{k[X_{r-h+1}, \ldots, X_r]/(X_{r-h+1}^{p}, \ldots, X_r^{p})}(N)_\lambda \neq 0$. By definition of the grading on the tensor product of modules, we get $\supp(\Omega^{i}_{B_{r-h}T}(M[\lambda])) \subseteq \supp(\Omega^{i}_{B_r T}(M \otimes_k N))$ for all $i > 0$. Let $\alpha = (1,-1)$ be the positive root of $G$ relative to $B$. By the proof of \ref{lem: tau for modules of complexity 1}, the Nakayama functor $\mathcal{N}$ of $\m B_r T$ is just the shift by $-(p^r -1) \alpha$. For every $s \in \N_0$, we get
\begin{equation*}
\supp(\tau_{B_r T}^{p^{s}}(M \otimes_k N)) \supseteq \supp(\Omega_{B_{r-h}T}^{2p^{s}}(M[\lambda])[- p^{s}(p^r-1)\alpha]).
\end{equation*}
By \cite[8.1.1]{Farn2}, there is $s \in \N_0$ such that 
\begin{equation*}
\Omega_{B_{r-h}T}^{2p^{s}}(M[\lambda])[- p^{s}(p^r-1)\alpha] = M[\lambda][- p^{s}(p^r-1)\alpha + p^{r-h} \alpha]
\end{equation*}
or 
\begin{equation*}
\Omega_{B_{r-h}T}^{2p^{s}}(M[\lambda])[- p^{s}(p^r-1)\alpha]= M[\lambda] [- p^{s}(p^r-1)\alpha].
\end{equation*}
Since similar computations can also be applied to $\tau^{i}_{B_r T}(M \otimes_k N) \cong \mathcal{N}^i \circ \Omega^{2i}_{B_r T}(M \otimes_k N)$ for $1 \leq i \leq p^s - 1$, arguments analogous to those in the proof of \ref{lem: cx(V)= 1 => tau - orbit contains only finitely many polynomial modules} show that the $\tau_{B_r T}$-orbit of $M \otimes_k N$ is finite. 
\end{proof}

Note that if $N$ in the situation of the previous result is projective, then $cx_{B_r T}(M \otimes_k N) =1$ and $\Theta \cap L_r D$ is finite by \ref{prop: only finitely many polynomial modules in component of cx 1}. If the component containing $M \otimes_k N$ is not regular, we have $cx_{U_r}(M \otimes_k N) = r$ and the Künneth-formula yields $h = r-1$ and $cx_{k[X_{r-h+1}, \ldots, X_r]/(X_{r-h+1}^{p}, \ldots, X_r^{p})}(N)=h$.

\begin{corollary}
Let $\lambda \in \Z^2$, $\Theta \subseteq \Gamma_s(B_r T)$ be the component containing $k_{\lambda}$. Then $\Theta \cap \m L_r D$ is finite.
\end{corollary}

\begin{proof}
For $r = 1$, this follows from \ref{prop: only finitely many polynomial modules in component of cx 1} since $cx_{B_r T}(k_{\lambda}) = 1$. For $r \geq 2$, $cx_{B_r T}(k_{\lambda}) = r > 1$, so that $k_{\lambda}$ is not periodic.  Now \cite[5.6]{Farn3} shows $F(\Theta) \cong \ZAinf$, so that $k_{\lambda}$ is a quasi-simple module in a $\ZAinf$-component of $\Gamma_s(B_r T)$. Suppose $\Theta$ is not regular. Then $F(\Theta)$ is not regular and the standard almost split sequence \cite[V.5.5]{ARS1} shows that $\Omega_{U_r}(k) \in F(\Theta)$. As $\Omega_{U_r}$ is an auto-equivalence of the stable module category, this shows that $\Omega_{U_r}(F(\Theta)) = F(\Theta)$, so that $\Omega_{U_r}$ defines an automorphism of $F(\Theta)$. In particular, $k$ and $\Omega_{U_r}(k)$ have the same quasi-length, so that $\tau_{U_r} = \Omega^2_{U_r}$ implies $\Omega(k)_ {U_r}^{2l+1}(k)=k$ for some $l \in \Z$. Thus, $k$ is periodic, a contradiction. Hence $\Theta$ is regular.
Since $U_r$ acts trivially on $k_{\lambda}$, we have $k_{\lambda} \cong k_{\lambda \vert B_1 T} \otimes_k k$, where we regard $k$ as the trivial module for $k[X_2, \ldots, X_r]/(X_2^p, \dots, X_r^p)$. The result now follows from \ref{prop: tensor product of cx 1 => quasilength of polynomial modules in component bounded; quasi-simple => only finitely many polynomial modules}.
\end{proof}

\section*{Acknowledgment}
The results of this article are part of my PhD thesis which I am currently writing at the University of Kiel. I would like to thank my advisor Rolf Farnsteiner for his support and helpful discussions as well as the members of our research team for proofreading.

\bibliographystyle{abbrv}
\bibliography{schur2}

\begin{thebibliography}{10}

\bibitem{ASS1}
I.~Assem, D.~Simson, and A.~Skowronski.
\newblock {\em {Elements of the Representation Theory of Associative Algebras,
  I}}.
\newblock Number~65 in London Mathematical Society Student Texts. Cambridge
  University Press, 2006.

\bibitem{ARS1}
M.~Auslander, I.~Reiten, and S.~Smalo.
\newblock {\em {Representation Theory of Artin Algebras}}.
\newblock Number~36 in Cambridge Studies in Advanced Mathematics. Cambridge
  University Press, 1995.

\bibitem{CR1}
C.~Curtis and I.~Reiner.
\newblock {\em {Methods of Representation Theory I}}.
\newblock Wiley, 1981.

\bibitem{DlRi1}
V.~Dlab and C.~M. Ringel.
\newblock The module-theoretical approach to quasi-hereditary algebras.
\newblock {\em London Math. Soc. Lecture Notes}, (168):200--224, 1992.

\bibitem{Doty1}
S.~Doty.
\newblock {Polynomial representations, algebraic monoids, and Schur algebras of
  classical type}.
\newblock {\em Journal of Pure and Applied Algebra}, (123):165--199, 1998.

\bibitem{DN1}
S.~Doty and D.~Nakano.
\newblock {Semisimple Schur Algebras}.
\newblock {\em Mathematical Proceedings of the Cambridge Philosophical
  Society}, (124):15--20, 1998.

\bibitem{DNP2}
S.~Doty, D.~Nakano, and K.~Peters.
\newblock {On infinitesimal Schur algebras}.
\newblock {\em Proc. London Math. Soc.}, (72):588--612, 1996.

\bibitem{DNP1}
S.~Doty, D.~Nakano, and K.~Peters.
\newblock {Polynomial representations of Frobenius kernels of $GL_{2}$}.
\newblock {\em Cont. Math.}, (194):57--67, 1996.

\bibitem{Dre1}
C.~Drenkhahn.
\newblock {On simple polynomial $G_r T$-modules}.
\newblock Link: http://arxiv.org/pdf/1601.03634v1.pdf.

\bibitem{Dre2}
C.~Drenkhahn.
\newblock {Fast zerfallende Folgen f{\"u}r Komodul-Algebren}.
\newblock Master's thesis, Christian-Albrechts-Universit{\"a}t zu Kiel, 2012.
\newblock Link:
  http://www.math.uni-kiel.de/algebra/de/drenkhahn/material/masterarbeit.

\bibitem{Erd1}
K.~Erdmann.
\newblock {\em {Blocks of Tame Representation Type and Related Classes of
  Algebras}}.
\newblock Number 1428 in Lecture Notes in Mathematics. Springer, 1990.

\bibitem{Erd2}
K.~Erdmann.
\newblock {On Auslander-Reiten components for group algebras}.
\newblock {\em Journal of Pure and Applied Algebra}, (104):149--160, 1995.

\bibitem{Farn3}
R.~Farnsteiner.
\newblock {On the Auslander-Reiten quiver of an Infinitesimal group}.
\newblock {\em Nagoya Math. J.}, 160:103--121, 2000.

\bibitem{Farn4}
R.~Farnsteiner.
\newblock {Auslander-Reiten Components for $G_1 T$-Modules}.
\newblock {\em Journal of Algebra and its applications}, 4(6):739--759, 2005.

\bibitem{Farn1}
R.~Farnsteiner.
\newblock {Group-graded algebras, extensions of infinitesimal groups, and
  applications}.
\newblock {\em Transform. Groups}, (14):127--162, 2009.

\bibitem{Farn2}
R.~Farnsteiner.
\newblock {Complexity, periodicity and one-parameter subgroups}.
\newblock {\em Trans. Amer. Math. Soc.}, (365):1487--1531, 2013.

\bibitem{FarnRoehr1}
R.~Farnsteiner and G.~R\"ohrle.
\newblock {Support varieties, AR-components and good filtrations}.
\newblock {\em Math. Z.}, (267):185--219, 2011.

\bibitem{FarnVoigt1}
R.~Farnsteiner and D.~Voigt.
\newblock {On Cocommutative Hopf Algebras of Finite Representation Type}.
\newblock {\em Advances in Mathematics}, (155):1--22, 2000.

\bibitem{GoGr2}
R.~Gordon and E.~L. Green.
\newblock {Graded Artin algebras}.
\newblock {\em Journal of Algebra}, (76):111--137, 1982.

\bibitem{GoGr1}
R.~Gordon and E.~L. Green.
\newblock {Representation Theory of Graded Artin Algebras}.
\newblock {\em Journal of Algebra}, (76):138--152, 1982.

\bibitem{G1}
J.~Green.
\newblock {On Certain Subalgebras of the Schur Algebra}.
\newblock {\em Journal of Algebra}, (131):265--280, 1990.

\bibitem{Ho1}
M.~Hoshino.
\newblock {On splittung torsion theories induced by tilting modules}.
\newblock {\em Communications in Algebra}, 11(4):427--439, 1983.

\bibitem{Jantz1}
J.~C. Jantzen.
\newblock {\em {Representations of Algebraic Groups}}.
\newblock Mathematical Surveys and Monographs, Vol. 107. American Mathematical
  Society, second edition edition, 2003.

\bibitem{Kuel1}
J.~K{\"u}lshammer.
\newblock {\em {Representation type and Auslander-Reiten theory of
  Frobenius-Lusztig kernels}}.
\newblock PhD thesis, Christian-Albrechts-Universit{\"a}t zu Kiel, 2012.

\bibitem{Nak1}
D.~Nakano.
\newblock {Varieties for $G_{r}T$-modules (with an appendix by J.C. Jantzen)}.
\newblock {\em Proc. of Symposia in Pure Mathematics}, pages 441--451, 1998.

\bibitem{Pre1}
A.~Premet.
\newblock {The Green ring of a simple three-dimensional Lie-p-algebra.}
\newblock {\em Soviet Mathematics}, (35 (10)):51--60, 1985.

\bibitem{Ren1}
L.~E. Renner.
\newblock {\em {Linear Algebraic Monoids}}.
\newblock Encyclopedia of Mathematical Sciences. Springer, 2005.

\bibitem{R1}
C.~M. Ringel.
\newblock {\em {Tame algebras and integral quadratic forms}}, volume 1099 of
  {\em Lecture Notes in Mathematics}.
\newblock Springer, 1984.

\end{thebibliography}
\end{document}